\documentclass[12pt]{amsart}
\usepackage{amsmath}
\usepackage{amsthm}
\usepackage{amssymb}
\usepackage{verbatim}
\usepackage{color}

\newtheorem{theorem}{Theorem}    
\newtheorem{proposition}[theorem]{Proposition}

\newtheorem{lemma}[theorem]{Lemma}

\newtheorem{remark}[theorem]{Remark}
\newtheorem{definition}[theorem]{Definition}
\theoremstyle{definition}

\numberwithin{theorem}{section} \numberwithin{theorem}{section}
\numberwithin{equation}{section}

\def\Rn{{\mathbb{R}^n}}

\def\supp{\operatorname{supp}}

\def\essi{\operatornamewithlimits{ess\,inf}}
\def\esss{\operatornamewithlimits{ess\,sup}}

\allowdisplaybreaks

\begin{document}
\title[Weighted variational inequalities for heat semigroups]
{Weighted variational inequalities for heat semigroups associated with Schr\"odinger operators related to critical radius functions}

\author{Yongming Wen$^{\ast}$ and Huoxiong Wu}


%
\thanks{$^*$Corresponding author.}
\thanks{Yongming Wen is supported by the National Natural Science Foundation of China (Grant No. 12301119), Fujian Alliance Of Mathematics (No. 2025SXLMQN07), President’s fund of Minnan Normal University (No. KJ2020020), Institute of Meteorological Big Data-Digital Fujian, Fujian Key Laboratory of Data Science and Statistics, Fujian Key Laboratory of Granular Computing and Applications (Minnan Normal University), China. Huoxiong Wu is supported by the National Natural Science Foundation of China (Grant Nos. 12171399, 12271041).}
\address{School of Mathematics and Statistics, Minnan Normal University, Zhangzhou 363000, China} \email{wenyongmingxmu@163.com}
\address{School of Mathematical Sciences, Xiamen University, Xiamen 361005, China} \email{huoxwu@xmu.edu.cn}



\begin{abstract}
Let $\mathcal{L}$ be a Schr\"{o}dinger operator and $\mathcal{V}_\varrho(e^{-t\mathcal{L}})$ be the variation operator of heat semigroup associated to $\mathcal{L}$ with $\varrho>2$. In this paper, we first obtain the quantitative weighted $L^p$ bounds for $\mathcal{V}_\varrho(e^{-t\mathcal{L}})$ with a class of weights related to critical radius functions, which contains the classical Muckenhoupt weights as a proper subset. Next, a new bump condition, which is weaker than the classical bump condition, 
is given for two-weight inequality of $\mathcal{V}_\varrho(e^{-t\mathcal{L}})$, and the weighted mixed weak type inequality corresponding to Sawyer's conjecture for $\mathcal{V}_\varrho(e^{-t\mathcal{L}})$ are obtained. Furthermore, the quantitative restricted weak type $(p,p)$ bounds for $\mathcal{V}_\varrho(e^{-t\mathcal{L}})$ are also given with a new class of weights $A_{p}^{\rho,\theta,\mathcal{R}}$, which is larger than the classical $A_{p}^{\mathcal{R}}$ weights. 
Meanwhile, several characterizations of $A_{p,q,\alpha}^{\rho,\theta,\mathcal{R}}$ in terms of restricted weak type estimates of maximal operators are established.
\end{abstract}

\maketitle

\section{Introduction and main results}
\subsection{Background}

Let $\{T_t\}_{t>0}$ be a family of operator bounded on some function spaces such that $\lim_{t\rightarrow0}T_tf$ exists in some sense. Variation operator is a substitute for a square function of the type
$$\Big(\sum_{i=1}^\infty|T_{t_i}f-T_{t_{i+1}}f|^2\Big)^{1/2}$$
to measure the speed of convergence of $\{T_t\}_{t>0}$. The study of variation operator was initiated by L\'{e}pengle \cite{Le}, and developed by Bourgain \cite{Bo}, who established variational estimates for ergodic averages of a dynamic system. Since then, the study of variation operator was taken into consideration in harmonic analysis, ergodic theory and probability and so on (for example, one can consult \cite{AJS,Jo,JR} etc.). Particularly, the breakthrough work of variation operators for singular integrals was established in \cite{CaJRW1,CaJRW2}, in which the authors gave the $L^p$-bounds and weak type $(1,1)$ bounds for variation of truncated singular integrals. For further results on variation operators, see \cite{BetFHR,BeHuWuYa,CheDHL,CheDHL2,CheHo,CheHoLi,CrMacMTV,DHL,DuLY,GT,HyLP,LW,MaTX1,MaTX2}, etc.

In this paper, we will focus on the variations of heat semigroups associated to differential operators. Given an operator $\mathcal{L}$, the $\varrho$-variation operators for heat semigroup associated with $\mathcal{L}$ is defined by
\begin{equation*}
\mathcal{V}_\varrho(e^{-t\mathcal{L}})f(x):=\sup_{\{t_j\}\downarrow0}
\Big(\sum_{j=1}^\infty|e^{-t_j\mathcal{L}}f(x)-e^{-t_{j+1}\mathcal{L}}f(x)|^\varrho\Big)^{1/\varrho},
~\varrho>2,
\end{equation*}
where the supreme is taken over all decreasing sequences $\{t_j\}$ that converge to 0. When $\mathcal{L}=-\triangle$, the Laplacian operator, Jones and Reinhold \cite{JR} and Crescimbeni et al. \cite{CrMacMTV} independently obtained weak type (1,1) bounds and $L^p$-bounds for $\mathcal{V}_{\varrho}({e^{-t\triangle}})$. The results in \cite{CrMacMTV,JR} were generalized to variation operators associated with approximate identities by Liu \cite{Liu}. Recently, Guo et al. \cite{GuWWY} extended the results of Liu \cite{Liu} to the weighted cases. When $\mathcal{L}=-\Delta+V$ is the Schr\"{o}dinger operator defined on $\mathbb{R}^n$ with $n\geq3$, where $V$ is a non-negative potential that belongs to $RH_s$, the reverse H\"older class, that is, there is a constant $C>0$ such that
$$\Big(\frac{1}{|B|}\int_BV(y)^sdy\Big)^{1/s}\leq\frac{C}{|B|}\int_BV(y)dy,~ s>n/2,$$
holds for each ball $B\subseteq \mathbb{R}^n$, Betancor et al. \cite{BetFHR} established weak type $(1,1)$ bounds and $L^p$-bounds for $\mathcal{V}_\varrho({e^{-t\mathcal{L}}})$; Tang and Zhang \cite{TZ} and Bui \cite{Bui} generalized the results of Betancor et al. in \cite{BetFHR} to the weighted cases and metric space, respectively; moreover, the current authors in \cite{WW} established quantitative weighted estimates for $\mathcal{V}_{\varrho}({e^{-t\mathcal{L}}})$ with weights adapted to Schr\"odinger setting, which contain the classical Muckenhoupt class as a proper subset (see subsection 2.2). Recently, by sparse dominations, Cao et al. \cite{CaSiZh} established several kinds of weighted inequalities for square functions associated to operators, which include the two-weight inequalities with bump conditions, Fefferman-Stein inequalities with arbitrary weights, the mixed weak type estimates corresponding to Sawyer's conjecture with the classical Muckenhoupt weights, and so on. Base on the above, it is natural to ask the following question.

\textbf{Question 1.} Can we obtain weighted inequalities for $\mathcal{V}_{\varrho}({e^{-t\mathcal{L}}})$ with weights adapted to Schr\"{o}dinger setting via sparse dominations, moreover, for the two-weight inequalities including bump conditions? for the mixed weak type estimates corresponding to Sawyer's conjecture with weights adapted to Schr\"{o}dinger setting?

Our first result is concerned with quantitative weighted estimates for $\mathcal{V}_{\varrho}({e^{-t\mathcal{L}}})$.
\begin{theorem}\label{quantitative Ap}
Let $n\geq3$, $V\in RH_{s}$ with $s>n/2$, and $\mathcal{L}=-\triangle+V$. Let $\varrho>2$, $\theta\geq0$ and $\rho$ be a critical radius function. Then for $1<p<\infty$ and $\omega\in A_p^{\rho,\theta}$,
\begin{equation*}
\|\mathcal{V}_\varrho(e^{-t\mathcal{L}})f\|_{L^p(\omega)}
\lesssim[\omega]_{A_p^{\rho,\theta}}^{\max\{1,1/(p-1)\}}\|f\|_{L^p(\omega)},
\end{equation*}
where $A_p^{\rho,\theta}$ is the weights class related to the critical radius function $\rho(x)$, which contains $A_p$, the classical Muckenhoupt class, as a proper subset (see subsection 2.2).
\end{theorem}

\begin{remark}
Comparing with Theorem 1.4's (i) in \cite{WW}, the result of Theorem \ref{quantitative Ap} is optimal because of that for any given $\theta>0$, $A_p^{\rho,\theta/p'}\subsetneqq A_p^{\rho,\theta}$. Moreover, our argument will be based on the sparse dominations, which is different from one in \cite{WW}.
\end{remark}

Next, we investigate two-weight inequality for variation operators. Two-weight inequalities for classical operators can be found in \cite{BonHQ,CruzM,CruzPe,Ka,Ler1,LRW,Perez,Sawyer,TV} and the references therein. In the following, we will give a weaker bump condition for two-weight inequality of $\mathcal{V}_{\varrho}({e^{-t\mathcal{L}}})$.
\begin{theorem}\label{quantitative two weight}
Let $n\geq3$, $V\in RH_{s}$ with $s>n/2$ and $\mathcal{L}=-\triangle+V$. For $\varrho>2$, $\theta\geq0$ and $\rho$ is the critical radius function associated with $\mathcal{L}$, if $\Phi,\Psi$ are Young functions such that $\bar{\Phi}\in B_{p'}$, $\bar{\Psi}\in B_p$ and a pair $(\mu,v)$ of weights satisfies
\begin{equation*}\label{bump condition}
[\mu,\nu]_{\Phi,\Psi,p,\rho,\theta}:=\sup_Q\|\nu^{1/p}\|_{\Phi,Q}\|\mu^{-1/p}\|_{\Psi,Q}
\psi_\theta(Q)^{-1}<\infty,
\end{equation*}
here and what follows, $\psi_\theta(Q):=(1+r_Q/\rho(x_Q))^\theta$. Then
\begin{equation*}
\|\mathcal{V}_\varrho(e^{-t\mathcal{L}})f\|_{L^p(\nu)}
\lesssim[\mu,\nu]_{\Phi,\Psi,p,\rho,\theta}[\bar{\Phi}]_{B_{p'}}^{1/p'}[\bar{\Psi}]_{B_p}^{1/p}
\|f\|_{L^p(\mu)}.
\end{equation*}
\end{theorem}

\begin{remark}Obviously, the bump condition in our Theorem \ref{quantitative two weight}
\begin{equation*}
\sup_Q\|\nu^{1/p}\|_{\Phi,Q}\|\mu^{-1/p}\|_{\Psi,Q}
\psi_\theta(Q)^{-1}<\infty,
\end{equation*}
is weaker than
\begin{equation*}
\sup_Q\|\nu^{1/p}\|_{\Phi,Q}\|\mu^{-1/p}\|_{\Psi,Q}<\infty,
\end{equation*}
which leads to the two-weight inequality for maximal truncated Calder\'on-Zygmund operator (see \cite[Theorem 1.2]{Ler1}).
\end{remark}

On the other hand, Berra, Pradolini and Quijano \cite{BePrQu} recently studied the following weighted mixed type inequalities for $M^{\rho,\sigma}$ and Schr\"{o}dinger Calder\'{o}n-Zygmund operator $T$.
\begin{theorem}\rm(\cite{BePrQu})\label{mixed weak type of CZ and HL}
Let $\rho$ be a critical radius function. Let $\mu\in A_1^\rho$ and $\nu\in A_\infty^\rho(\mu)$. Then for any $\lambda>0$,
\begin{align*}
\lambda\mu\nu\Big(\Big\{x\in\mathbb{R}^n:\frac{|S(f\nu)(x)|}{\nu(x)}>\lambda\Big\}\Big)\lesssim
\int_{\mathbb{R}^n}|f(x)|\mu(x)\nu(x)dx,
\end{align*}
where $S$ is a Schr\"{o}dinger Calder\'{o}n-Zygmund operator $T$ or $M^{\rho,\sigma}$ for some $\sigma\geq0$.
\end{theorem}

It is well-known that the investigation on the weighted mixed weak type inequality for the classical operators in harmonic analysis originated from the work of  Sawyer in \cite{Sawyer1},
which can provide an alternative proof of weighted $L^p$-boundedness of Hardy-Littlewood maximal operator $M$ with $A_p$ weights. Precisely, Sawyer \cite{Sawyer1} obtained the following weighted mixed weak type inequality.

\begin{theorem}\rm(\cite{Sawyer1})
Let $\mu,\nu\in A_1$. Then
\begin{align}\label{general mix}
\lambda\mu\nu\Big(\Big\{x\in\mathbb{R}:\frac{M(f\nu)(x)}{\nu(x)}>\lambda\Big\}\Big)
\lesssim\int_{\mathbb{R}}|f(x)|\mu(x)\nu(x)dx.
\end{align}
\end{theorem}

Subsequently, Cruz-Uribe et al. \cite{CrMP} extended \eqref{general mix} to higher dimension as well as the corresponding weighted mixed weak type inequality for Calder\'{o}n-Zygmund operator, and Li et al. \cite{LiOP} also obtained the corresponding results under weakening the assumptions in \cite{CrMP}. Ones can consult \cite{CalRi,CaIRXY,CaSiZh,CaXYa,LiOB,WSS} etc. for more development of weighted mixed weak type inequalities.

Inspired by Theorem \ref{mixed weak type of CZ and HL} above, we aim to establish the weighted mixed weak type inequality for $\mathcal{V}_\varrho(e^{-t\mathcal{L}})$.



\begin{theorem}\label{mixed weak type inequality for variation of heat semigroup in Sch}
Let $n\geq3$, $V\in RH_{s}$ with $s>n/2$, and $\mathcal{L}=-\triangle+V$. Let $\varrho>2$ and $\rho$ be a critical radius function. Suppose that $\mu\in A_1^\rho$ and $\nu\in A_\infty^\rho(\mu)$. Then for any $\lambda>0$,
\begin{align*}
\lambda\mu\nu\Big(\Big\{x\in\mathbb{R}^n:\frac{\mathcal{V}_\varrho(e^{-t\mathcal{L}})(f\nu)(x)}{\nu(x)}>
\lambda\Big\}\Big)\lesssim\int_{\mathbb{R}^n}|f(x)|\mu(x)\nu(x)dx.
\end{align*}
\end{theorem}

\begin{remark}Our original idea is to use Calder\'{o}n-Zygmund decomposition like \cite{BePrQu} to prove our theorem. Unfortunately, it is out of our expectation, since the argument based on this idea comes to an end when we try to use the cancellation of bad function and the perturbation formula in \cite{Dz} because of that the perturbation formula targets two spatial variables, not three different spatial variables. Hence, to obtain our Theorem \ref{mixed weak type inequality for variation of heat semigroup in Sch}, we need to find new techniques to tackle weighted mixed weak type inequalities in the Schr\"{o}dinger setting, which will be presented in Section 4.
\end{remark}

In addition, for $1\leq p<\infty$, Kerman and Torchinsky \cite{KT} introduced a class of weights $A_p^\mathcal{R}$, in which the elements satisfy
\begin{equation*}
\|\omega\|_{A_p^\mathcal{R}}:=\sup_{E\subset Q}\frac{|E|}{|Q|}\Big(\frac{\omega(Q)}{\omega(E)}\Big)^{1/p}<\infty.
\end{equation*}
We mention that $A_p\subseteq A_p^\mathcal{R}$ for $1<p<\infty$ and $A_1=A_1^\mathcal{R}$. $A_p^\mathcal{R}$ class was used in \cite{KT} to characterize the following restricted weak type $(p,p)$ estimate of the Hardy-Littlewood maximal operator $M$,
\begin{align*}
M:L^{p,1}(\omega)\rightarrow L^{p,\infty}(\omega)\Leftrightarrow \omega\in A_p^{\mathcal{R}},
\end{align*}
moreover,
\begin{align*}
\|M\chi_E\|_{L^{p,\infty}(\omega)}\lesssim\|\omega\|_{A_p^\mathcal{R}}\omega(E)^{1/p}.
\end{align*}
Let $\theta\geq0$ and $0\leq \alpha<n$. Define fractional maximal function $M_\alpha^{\rho,\theta}$ associated with critical radius function $\rho$ as below,
\begin{align*}
M_\alpha^{\rho,\theta}f(x):=\sup_{Q\ni x}\frac{1}{(\psi_\theta(Q)|Q|)^{1-\frac{\alpha}{n}}}\int_Q|f(y)|dy.
\end{align*}
For brevity, we denote $M_0^{\rho,\theta}$ by $M^{\rho,\theta}$. $M_\alpha^{\rho,\theta}$ is an vital operator to establish the boundedness of singular integrals associated with differential operators, for instance, see \cite{BonHQ,BuiBD1,BuiBD,LRW,TZ,WW,ZhYa}. It is very natural to consider what types of weights can be used to characterize the restricted weak type estimates of $M_\alpha^{\rho,\theta}$. Obviously, when $\alpha=0$, $A_p^\mathcal{R}$ is not the right class of weights that we seek. Based on the above, we raise the following question.

\textbf{Question 2.} What types of weights can be used to characterize the restricted weak type estimates of $M_\alpha^{\rho,\theta}$? How to deal with this new class of weights to obtain the quantitative restricted weak type estimates of $M_\alpha^{\rho,\theta}$ and $\mathcal{V}_\varrho(e^{-t\mathcal{L}})$?

To solve this question, we introduce a new class of weights, denoted by $A_{p,q,\alpha}^{\rho,\theta,\mathcal{R}}$, where the definition is given in Section 2 (see Definition \ref{a new class of weights}). The following theorem is concerned with the characterization of this new class of weights.
\begin{theorem}\label{characterization restricted weak type}
Let $\theta\geq0$, $0\leq\alpha<n$, $1\leq p<\infty$ and $p\leq q$. Let $\mu$, $\nu$ be weights and $\rho$ be a critical radius function. Then for every measurable function $f$,
\begin{equation}\label{restricted weak type for maximal function}
\|M_\alpha^{\rho,\theta}f\|_{L^{q,\infty}(\nu)}\leq c\|f\|_{L^{p,1}(\mu)}
\end{equation}
if and only if $(\mu,\nu)\in A_{p,q,\alpha}^{\rho,\theta,\mathcal{R}}$. Moreover, if $(\mu,\nu)\in A_{p,q,\alpha}^{\rho,\theta,\mathcal{R}}$, then
\begin{equation*}
\|M_\alpha^{\rho,\theta}f\|_{L^{q,\infty}(\nu)}\lesssim [\mu,\nu]_{A_{p,q,\alpha}^{\rho,\theta,\mathcal{R}}}\|f\|_{L^{p,1}(\mu)}.
\end{equation*}
\end{theorem}
\begin{remark}
From the definition of $A_{p,q,\alpha}^{\rho,\theta,\mathcal{R}}$, if $\mu,v\equiv1$, it is not hard to see that $[\mu,\nu]_{A_{p,q,\alpha}^{\rho,\theta,\mathcal{R}}}<\infty$ if and only if $1/p-1/q=\alpha/n$.
\end{remark}

The next theorem indicates the equivalence between $[\mu,v]_{A_{p,q,\alpha}^{\rho,\theta,\mathcal{R}}}$ and $\|\mu,v\|_{A_{p,q,\alpha}^{\rho,\theta,\mathcal{R}}}$, where the latter is useful in revealing the relationship between $A_{p}$ and $A_p^{\rho,\theta,\mathcal{R}}$, see Section 2 for details.
\begin{theorem}\label{another characterization restricted weak type}
Let $\theta\geq0$, $0\leq\alpha<n$, $1\leq p<\infty$ and $p\leq q$. Let $\mu$, $\nu$ be weights and $\rho$ be a critical radius function. The following statements are equivalent:
\begin{itemize}
\item [(1)]$\|M_\alpha^{\rho,\theta}f\|_{L^{q,\infty}(\nu)}\lesssim\|f\|_{L^{p,1}(\mu)}$;
\item [(2)]$\|M_\alpha^{\rho,\theta}(\chi_E)\|_{L^{q,\infty}(\nu)}\lesssim\mu(E)^{1/p}$ for any measurable set $E\subseteq\mathbb{R}^n$;
\item [(3)]
$
\|\mu,\nu\|_{A_{p,q,\alpha}^{\rho,\theta,\mathcal{R}}}:=\sup_Q\nu(Q)^{1/q}
    \sup_{0<\mu(P)<\infty}\frac{|P\cap Q|}{(|Q|\psi_\theta(Q))^{1-\frac{\alpha}{n}}}\mu(P)^{-1/p}
    <\infty;
$
\item [(4)]$(\mu,\nu)\in A_{p,q,\alpha}^{\rho,\theta,\mathcal{R}}$.
\end{itemize}
\end{theorem}

In the following, if $p=q$ and $\alpha=0$, we write $(\omega,\omega)\in A_{p,p,0}^{\rho,\theta,\mathcal{R}}$ by $\omega\in A_{p}^{\rho,\theta,\mathcal{R}}$ for brevity. The final theorem is related to the quantitative restricted weak type estimates of $\mathcal{V}_\varrho(e^{-t\mathcal{L}})$.

\begin{theorem}\label{quantitative restricted weak type}
Let $n\geq3$, $V\in RH_{s}$ with $s>n/2$ and $\mathcal{L}=-\triangle+V$. Let $\varrho>2$, $\theta\geq0$ and $\rho$ be a critical radius function. If $\omega\in A_{p}^{\rho,\theta,\mathcal{R}}$ with $1<p<\infty$, then
$$\|\mathcal{V}_\varrho(e^{-t\mathcal{L}})f\|_{L^{p,\infty}(\omega)}
\lesssim[\omega]_{A_{p}^{\rho,\theta,\mathcal{R}}}^{p+1}\|f\|_{L^{p,1}(\omega)}.$$
\end{theorem}
\begin{remark}
Comparing with the classes of weights in \cite{CaSiZh}, the classes of weights we consider are larger.
\end{remark}

We organize the rest of the paper as follows. In Section 2, we will give some preliminaries. In Section 3, we prove Theorems \ref{quantitative Ap}-\ref{quantitative two weight}. The proof of Theorem \ref{mixed weak type inequality for variation of heat semigroup in Sch} is given in Section 4. The proofs of Theorem \ref{characterization restricted weak type} and Theorem \ref{another characterization restricted weak type} will be given in Section 5. Finally, we present the proof of Theorem \ref{quantitative restricted weak type} in Section 6.

Throughout the rest of the paper, we denote $f\lesssim g$, $f\thicksim g$ if $f\leq Cg$ and $f\lesssim g \lesssim f$, respectively. For any cube $Q:=Q(x_Q,r_Q)\subseteq \mathbb{R}^n$ and $\sigma>0$, $x_Q$ and $r_Q$ are the center and side-length of $Q$, respectively, $\chi_Q$ represents the characteristic function of $Q$ and $\sigma Q$ means $Q(x_Q,\sigma r_Q)$, $\langle f\rangle_Q$ represents the average $|Q|^{-1}\int_Qf(x)dx$ of $f$ over the cube $Q$.

\section{Preliminaries}
In this section, we introduce some classes of weights and recall some necessary definitions and lemmas.
\subsection{Lorentz spaces}
Let $p>0$ and $(\mathbb{R}^n,\nu)$ be a measure space. $L^{p,1}(\nu)$ is the space of $\nu$-measurable functions such that
\begin{equation*}
\|f\|_{L^{p,1}(\nu)}:=p\int_0^\infty \nu(\{x\in\mathbb{R}^n:|f(x)|>y\})^{1/p}dy.
\end{equation*}
If $p\geq1$, it is known that $L^{p,1}(\nu)\hookrightarrow L^p(\nu)\hookrightarrow L^{p,\infty}(\nu)$. We end this subsection by recalling the following summation lemma,
\begin{lemma}{\rm(\cite{R})}\label{summation lemma}
Let $0<p<\infty$ and $\gamma\geq\max\{p,1\}$. Given a measurable function $f$ and a weight $\omega$, if $\{E_j\}$ is a collection of measurable sets such that $\sum_{j\geq1}\chi_{E_j}\leq C$, then
\begin{equation*}
\sum_{j\geq1}\|f\chi_{E_j}\|_{L^{p,1}(\omega)}^\gamma\leq C^{\gamma/p}\|f\|_{L^{p,1}(\omega)}^\gamma.
\end{equation*}
\end{lemma}
\subsection{Weights}
Recall that $\rho$ is a critical radius function if there exist $C_0,N_0>0$ such that for any $x,y\in\mathbb{R}^n$,
\begin{align}\label{critical radius function}
C_0^{-1}\rho(x)\Big(1+\frac{|x-y|}{\rho(x)}\Big)^{-N_0}\leq\rho(y)\leq C_0\rho(x)
\Big(1+\frac{|x-y|}{\rho(x)}\Big)^{\frac{N_0}{N_0+1}}.
\end{align}

\begin{definition}\label{def2.1}
A weight is a locally integrable function on $\mathbb{R}^n$ that takes values in $(0,\infty)$ almost everywhere. Let $\theta\geq0$, $\rho$ be a critical radius function and $\mu$ be a weight. We say that $w\in A_{p}^{\rho,\theta}(\mu)~(1<p<\infty)$ if
\begin{align*}
[\omega]_{A_p^{\rho,\theta}(\mu)}
:=\sup_{Q}\Big(\frac{1}{\mu(Q)}
\int_{Q}\omega(x)\mu(x)dy\Big)\Big(\frac{1}{\mu(Q)}
\int_{Q}\omega(x)^{1-p'}\mu(x)dx\Big)^{p-1}\psi_\theta(Q)^{-1}<\infty.
\end{align*}
We say that $w\in A_{1}^{^{\rho,\theta}}(\mu)$ if
$$[\omega]_{A_{1}^{^{\rho,\theta}}(\mu)}:=\sup_Q\frac{1}{\psi_\theta(Q)\mu(Q)}\int_{Q}
\omega(x)\mu(x)dx(\essi_{x\in Q}\omega(x))^{-1}<\infty,$$
where the supremum is taken over all cubes $Q:=Q(x_Q,r_Q)\subseteq\mathbb{R}^n$. We also define
$$A_p^{\rho}(\mu):=\bigcup_{\theta\geq0}A_p^{\rho,\theta}(\mu),\quad A_\infty^{\rho}(\mu):=\bigcup_{p\geq1}A_p^{\rho}(\mu).$$
\end{definition}
There are several characterizations of $A_\infty^{\rho}(\mu)$, here, we use the following characterization: $\omega\in A_\infty^{\rho}(\mu)\Leftrightarrow$ there are $\theta\geq0$ and $\epsilon>0$ such that
\begin{align*}
\frac{\omega\mu(E)}{\omega\mu(Q)}\lesssim\psi_\theta(Q)\Big(\frac{\mu(E)}{\mu(Q)}\Big)^\epsilon
\end{align*}
holds for each cube $Q$ and every measurable set $E\subseteq Q$. For brevity, we denote $A_p^{\rho,\theta}(1)$ by $A_p^{\rho,\theta}$ when $\mu=1$. Specially, $A_p$ weights coincide $A_p^{\rho,0}$ weights. It is obvious that the class of Muckenhoupt $A_p$ weights is a subset of the class of $A_p^{\rho,\theta}$ weights. An important property of $A_p^\rho$ weights is the reverse H\"{o}lder inequality that they satisfy.

\begin{definition}
Let $\theta\geq0$, $\rho$ be a critical radius function. We say that a weight $\omega\in RH_s^{\rho,\theta}$ $(1<s<\infty)$ if
\begin{align*}
\Big(\frac{1}{|Q|}\int_Q\omega(x)^sdx\Big)^{1/s}\lesssim\psi_\theta(Q)
\Big(\frac{1}{|Q|}\int_Q\omega(x)dx\Big).
\end{align*}
We say $\omega\in RH_\infty^{\rho,\theta}$ if
\begin{align*}
\esss_{x\in Q}\omega(x)\lesssim\psi_\theta(Q)
\Big(\frac{1}{|Q|}\int_Q\omega(x)dx\Big).
\end{align*}
We also define $RH_s^\rho:=\bigcup_{\theta\geq0}RH_s^{\rho,\theta}$, $1<s\leq\infty$.
\end{definition}

We will use the following properties of $RH_s^\rho$ classes.
\begin{lemma}{\rm(\cite{BePrQu})}\label{property of reverse holder class}
Let $\rho$ be a critical radius function.\\
$(1)$~Let $\omega\in A_p^\rho\cap RH_s^\rho$. Then there exist weight $\omega_1$, $\omega_2$ such that $\omega=\omega_1\omega_2$, where $\omega_1\in A_1^\rho\cap RH_s^\rho$, $\omega_2\in A_p^\rho\cap RH_\infty^\rho$;\\
$(2)$~Let $\omega_1$, $\omega_2\in RH_\infty^\rho$. Then $\omega_1\omega_2\in RH_\infty^\rho$.
\end{lemma}

Let $\rho$ be a critical radius function. Denote
$$\mathcal{B}_\rho:=\{B(x,r):x\in\mathbb{R}^n,r\leq\rho(x)\}.$$
The following class of $\rho$-localized weights is introduced in \cite{BonHS}.
\begin{definition}\label{local weight}
Let $1<p<\infty$. We say a weight $\omega\in A_{p}^{\rho,loc}$ if
\begin{equation*}
[\omega]_{A_{p}^{\rho,loc}}:=\sup_{B\in\mathcal{B}_\rho}\Big(\frac{1}{|B|}\int_B
\omega(x)dx\Big)\Big(\frac{1}{|B|}\int_B\omega(x)^{-1/(p-1)}dx\Big)^{p-1}<\infty,
\end{equation*}
and we say that $\in A_{1}^{\rho,loc}$ if
\begin{equation*}
[\omega]_{A_{1}^{\rho,loc}}:=\sup_{B\in\mathcal{B}_\rho}\Big(\frac{1}{|B|}\int_B
\omega(x)dx\Big)\|\omega^{-1}\|_{L^\infty(B)}<\infty,
\end{equation*}
where the supremum is taken over all balls $B\subseteq\mathbb{R}^n$. We also define $A_\infty^{\rho,loc}=\bigcup_{p\geq1}A_p^{\rho,loc}$.
\end{definition}
We give a remark about this class of weights.
\begin{remark}\label{remark 2.4}
Let $\theta\geq0$ and $1\leq p<\infty$.\\
$(1)$~It is direct that $A_{p}^{\rho,\theta}\subseteq A_{p}^{\rho,loc}$ and $[\omega]_{A_{p}^{\rho,loc}}\lesssim[\omega]_{A_{p}^{\rho,\theta}}$; \\
$(2)$~For any $\beta>1$, it was proved in \cite{BonHS} that $A_p^{\beta\rho,loc}=A_p^{\rho,loc}$ and $$[\omega]_{A_p^{\rho,loc}}\sim[\omega]_{A_p^{\beta\rho,loc}}.$$
\end{remark}

Let $B_0$ be a ball and $\omega$ be defined on $B_0$. We say that a weight $\omega\in A_p(B_0)$ with $1\leq p<\infty$ if $B\in\mathcal{B}_\rho$ is replaced by $B\subseteq B_0$ in Definition \ref{local weight}. The following lemma is related to the extension of weights.

\begin{lemma}{\rm(\cite{BonHS})}\label{extension}
Let $1\leq p<\infty$. Given a ball $B_0$ in $\mathbb{R}^n$ and a weight $\omega_0\in A_p(B_0)$, then $\omega_0$ has an extension $\omega\in A_p$ such that for any $x\in B_0$, $\omega_0(x)=\omega(x)$ and $[\omega]_{A_p(B_0)}\sim[\omega]_{A_p}$, where the implicit constants are independent of $\omega_0$ and $p$.
\end{lemma}

To characterize the restricted weak type estimates of $M_\alpha^{\rho,\theta}$, we introduce a new class of weights $A_{p,q,\alpha}^{\rho,\theta,\mathcal{R}}$ as follows.
\begin{definition}\label{a new class of weights}
Let $\theta\geq0$, $0\leq\alpha<n$, $1\leq p<\infty$, $p\leq q$ and $\rho$ be a critical radius function. Let $(\mu,\nu)$ be a pair of weights. We say that $(\mu,\nu)\in A_{p,q,\alpha}^{\rho,\theta,\mathcal{R}}$ if
\begin{equation*}
[\mu,\nu]_{A_{p,q,\alpha}^{\rho,\theta,\mathcal{R}}}:=
\sup_Q\nu(Q)^{1/q}\frac{\|\chi_Q\mu^{-1}\|_{L^{p',\infty}(\mu)}
}{(|Q|\psi_\theta(Q))^{1-\frac{\alpha}{n}}}<\infty.
\end{equation*}
\end{definition}
To have a better understanding of this new class of weights, we give the following theorem to illustrate the relationship between $A_{p}^{\rho}$ and $A_{p}^{\rho,\mathcal{R}}:=\bigcup_{\theta\geq0}A_{p}^{\rho,\theta,\mathcal{R}}$.
\begin{theorem}
Let $1\leq p<q<\infty$. Then $A_{1}^{\rho}= A_{1}^{\rho,\mathcal{R}}$, $A_{p}^{\rho}\subseteq A_{p}^{\rho,\mathcal{R}}\subseteq A_{q}^{\rho}$ and $A_\infty^{\rho}=\bigcup_{p\geq1}A_{p}^{\rho}=\bigcup_{p\geq1}A_{p}^{\rho,\mathcal{R}}$. Moreover,
\begin{equation*}
[\omega]_{A_{p}^{\rho,\theta,\mathcal{R}}}\leq[\omega]_{A_{p}^{\rho,p\theta}}^{1/p}~and~
[\omega]_{A_{q}^{\rho,q\theta}}\leq\Big(\frac{p'}{p'-q'}\Big)^{q-1}
[\omega]_{A_{p}^{\rho,\theta,\mathcal{R}}}^q.
\end{equation*}
\end{theorem}
\begin{proof}
It is obvious that $A_{1}^{\rho}= A_{1}^{\rho,\mathcal{R}}$. To prove the rest of the theorem, it suffices to show
\begin{equation*}
[\omega]_{A_{p}^{\rho,\theta,\mathcal{R}}}\leq[\omega]_{A_{p}^{\rho,p\theta}}^{1/p},~
[\omega]_{A_{q}^{\rho,q\theta}}\leq\Big(\frac{p'}{p'-q'}\Big)^{q-1}[\omega]_{A_{p}^{\rho,\theta,\mathcal{R}}}.
\end{equation*}
For any measurable set $E\subset Q$,
\begin{align*}
|E|=\int_E\omega(x)^{1/p}\omega(x)^{-1/p}dx\leq\omega(E)^{1/p}\Big(\int_Q\omega(x)^{1-p'}dx\Big)^{1/p'}.
\end{align*}
Then
\begin{align*}
\frac{|E|}{|Q|\psi_\theta(Q)}\Big(\frac{\omega(Q)}{\omega(E)}\Big)^{1/p}&\leq
\frac{\omega(Q)^{1/p}}{|Q|\psi_\theta(Q)}\Big(\int_Q\omega(x)^{1-p'}dx\Big)^{1/p'}\\
&=\Big(\frac{1}{|Q|\psi_\theta(Q)}\int_Q\omega(x)dx\Big)^{1/p}
\Big(\frac{1}{|Q|\psi_\theta(Q)}\int_Q\omega(x)^{1-p'}dx\Big)^{1/p'}\\
&\leq[\omega]_{A_{p}^{\rho,p\theta}}^{1/p}.
\end{align*}
It follows that $\|\omega\|_{A_{p}^{\rho,\theta,\mathcal{R}}}\leq[\omega]_{A_{p}^{\rho,p\theta}}^{1/p}$. From the proof of Theorem \ref{another characterization restricted weak type}, we in fact prove that $[\omega]_{A_{p}^{\rho,\theta,\mathcal{R}}}\leq\|\omega\|_{A_{p}^{\rho,\theta,\mathcal{R}}}$. Hence, we have $[\omega]_{A_{p}^{\rho,\theta,\mathcal{R}}}\leq[\omega]_{A_{p}^{\rho,p\theta}}^{1/p}$.

In addition, for any cube $Q$, using \cite[Exercise 1.1.11]{Gra}, we have
\begin{equation*}
\int_Q(\omega(x)^{-1})^{q'}\omega(x)dx\leq\frac{p'}{p'-q'}\omega(Q)^{1-\frac{q'}{p'}}
\|\chi_Q\omega^{-1}\|_{L^{p',\infty}(\omega)}^{q'},~0<q'<p'.
\end{equation*}
This, together with $q'(q-1)=q$, allows us to get that
\begin{align*}
&\Big(\frac{1}{|Q|\psi_\theta(Q)}\int_Q\omega(x)dx\Big)
\Big(\frac{1}{|Q|\psi_\theta(Q)}\int_Q\omega(x)^{-\frac{1}{q-1}}dx\Big)^{q-1}\\
&\quad=\frac{\omega(Q)}{|Q|^q\psi_\theta(Q)^q}\Big(\int_Q(\omega(x)^{-1})^{q'}\omega(x)dx\Big)^{q-1}\\
&\quad\leq\Big(\frac{p'}{p'-q'}\Big)^{q-1}
\frac{\omega(Q)^{\frac{q}{p}}\|\chi_Q\omega^{-1}\|_{L^{p',\infty}(\omega)}^q}{(|Q|\psi_\theta(Q))^q}\\
&\quad\leq\Big(\frac{p'}{p'-q'}\Big)^{q-1}[\omega]_{A_{p}^{\rho,\theta,\mathcal{R}}}^q.
\end{align*}
Therefore, we have prove that $[\omega]_{A_{q}^{\rho,q\theta}}\leq\Big(\frac{p'}{p'-q'}\Big)^{q-1}
[\omega]_{A_{p}^{\rho,\theta,\mathcal{R}}}^q$.
\end{proof}

\subsection{Sparse operators associated to critical radius functions}
In this subsection, we introduce the power tool sparse operator that we use. Recall that a collection of cubes $\mathcal{D}$ in $\mathbb{R}^d$ is called a dyadic grid if it satisfies:\\
{\rm (1)} for any $Q\in\mathcal{D}$, $l(Q)=2^k$ for some $k\in\mathbb{Z}$;\\
{\rm (2)} for any $Q_1, Q_2\in\mathcal{D}$, $Q_1\cap Q_2=\{Q_1,Q_2,\emptyset\}$;\\
{\rm (3)} for any $k\in\mathbb{Z}$, the family $\mathcal{D}_k=\{Q\in\mathcal{D}:l(Q)=2^k\}$ forms a partition of $\mathbb{R}^d$.

An important sub-family of dyadic grids is sparse family, where its definition is given below.

\begin{definition}\rm(\cite{Ler3})
Let $\eta\in(0,1)$. A subset $\mathcal{S}\subseteq\mathcal{D}$ is said to be a $\eta$-sparse family with $\eta\in(0,1)$ if for any cube $Q\in\mathcal{S}$, there exists a measurable subset $E_Q\subseteq Q$ such that $\eta|Q|\leq|E_Q|$, and the sets $\{E_Q\}_{Q\in\mathcal{S}}$ are mutually disjoint.
\end{definition}

Let $\rho$ be a critical radius function and $\mathcal{S}$ be a sparse family. Define
\begin{equation*}
\mathcal{A}_{\mathcal{S}}^{\rho,N}f(x):=\sum_{Q\in\mathcal{S}}\Big(\frac{1}{|3Q|}\int_{3Q}|f(y)|dy\Big)
\psi_N(Q)^{-1}\chi_Q(x),~N\geq0.
\end{equation*}
In \cite{BuiBD1}, the authors proved the following estimate for $\mathcal{A}_{\mathcal{S}}^{\rho,N}$.
\begin{lemma}\rm(\cite{BuiBD1})\label{weighted estimates for sparse operator}
Suppose that $\omega\in A_{p}^{\rho,\theta}$ for $1<p<\infty$ and $\theta\geq0$. Then
\begin{equation*}
\|\mathcal{A}_{\mathcal{S}}^{\rho,N}f\|_{L^p(\omega)}\lesssim[\omega]_{A_{p}^{\rho,\theta}}
^{\max\{1,\frac{1}{p-1}\}}\|f|_{L^p(\omega)}
\end{equation*}
provided that $N\geq\theta\max\{1,\frac{1}{p-1}\}$.
\end{lemma}
We end this subsection by recalling a principle of sparse domination.
\begin{proposition}{\rm(\cite{Ler1})}\label{principle}
Let $T$ be a sublinear operator which is of weak type $(r,r)$ with $1\leq r<\infty$. Suppose that $\mathcal{M}_T$ is of weak type $(r,r)$. Then there is a constant $C>0$ such that for all $f\in L^r$ supported in $3Q_0$ for some $Q_0\in\mathcal{D}$, there is a sparse family $\mathcal{S}\subseteq\mathcal{D}(Q_0)$(depending on $f$) such that
\begin{equation*}
|Tf|\chi_{Q_0}\leq C\sum_{P\subset\mathcal{S}}\langle|f|^r\rangle_{3P}^{1/r}\chi_P.
\end{equation*}
Here $\mathcal{D}(Q_0)$ is the set of all dyadic cubes with respect to $Q_0$,
\begin{equation*}
\mathcal{M}_Tf(x):=\sup_{Q\ni x}\|T(f\chi_{\mathbb{R}^n\backslash3Q})\|_{L^\infty(Q)},
\end{equation*}
where the supremum is taken over all cubes $Q\subset\mathbb{R}^n$ containing $x$.
\end{proposition}
\subsection{Orlicz maximal operators} Recall that $\Phi:[0,\infty)\rightarrow[0,\infty)$ is called a Young function if it is continuous, increasing, convex and satisfies $\Phi(0)=0$ and $\lim_{t\rightarrow\infty}\Phi(t)/t=\infty$. The corresponding complementary function of $\Phi$, denoted by $\bar{\Phi}:[0,\infty)\rightarrow[0,\infty)$, is given by
\begin{equation*}
\bar{\Phi}(t)=\sup_{s>0}\{st-\Phi(s)\}.
\end{equation*}
Let $\Phi$ be a Young function. The localized Orlicz norm $\|f\|_{\Phi,Q}$ is defined by
\begin{equation*}
\|f\|_{\Phi,Q}:=\inf\Big\{\lambda>0:\frac{1}{|Q|}\int_Q\Phi\Big(\frac{|f(x)|}{\lambda}\Big)dx\leq1\Big\}.
\end{equation*}
Then the Orlicz maximal operator $M_\Phi$ is defined by
\begin{equation*}
M_\Phi f(x):=\sup_{Q\ni x}\|f\|_{\Phi,Q}.
\end{equation*}
The $L^p$-boundedness of $M_\Phi$ can be characterized via the $B_p$ condition, that is,
\begin{lemma}{\rm(cf. \cite{CMP11,HyPe})}\label{boundedness of Orlicz maximal operator}
Let $1<p<\infty$. $M_\Phi$ is bounded on $L^p(\mathbb{R}^n)$ if and only if $\Phi\in B_p$. Moreover, $\|M_\Phi\|_{L^p(\mathbb{R}^n)\rightarrow L^p(\mathbb{R}^n)}\leq C_{n,p}[\Phi]_{B_p}^{1/p}$, where
\begin{equation*}
[\Phi]_{B_p}:=\int_1^\infty\frac{\Phi(t)}{t^p}\frac{dt}{t}<\infty.
\end{equation*}
\end{lemma}
\subsection{Lemmata}
In this subsection, we recall some useful lemmas. The first lemma is the quantitative weighted bounds and two-weight inequality of Hardy-Littlewood maximal operators $M^{\rho,\theta}$ associated with critical radius function.
\begin{lemma}{\rm(\cite{BonHQ,BuiBD,WSS})}\label{quantitative weighted bounds for Hardy-Littlewood maximal operators}
Let $\rho$ be critical radius function, $\theta\geq0$ and $1<p<\infty$. \\
$(1)$~$\|M^{\rho,\theta} f\|_{L^p(\omega)}\lesssim [\omega]_{A_{p}^{\rho,\theta(p-1)}}^{1/(p-1)}\|f\|_{L^p(\omega)}$.\\
$(2)$~Let $1<p<\infty$, $\theta\geq0$. If $\Psi$ is a Young function such that $\bar{\Psi}\in B_{p}$ and a pair $(\mu,\nu)$ of weights satisfies
\begin{equation*}
[\mu,\nu]_{\Psi,p,\rho,\theta}:=\sup_Q\|\nu^{1/p}\|_{p,Q}\|\mu^{-1/p}\|_{\Psi,Q}\psi_\theta(Q)^{-1}<\infty.
\end{equation*}
Then there exists $\sigma_0$ such that for all $\sigma\geq \sigma_0+\sigma_0(\theta)$,
\begin{equation*}
\|M^{\rho,\sigma}f\|_{L^p(\nu)}\lesssim[\mu,\nu]_{\Psi,p,\rho,\theta}[\bar{\Psi}]_{B_p}^{1/p}
\|f\|_{L^p(\mu)}.
\end{equation*}
\end{lemma}

Finally, we state the covering lemma in \cite{DzZ}.
\begin{lemma}{\rm(\cite{DzZ})}\label{covering lemma}
There exists a sequence of points $x_j$ in $\Rn$, so that the family $\{B_j:=B(x_j,\rho(x_j))\}_{j\in\mathbb{Z}^+}$ satisfies:\\
$(i)$~$\bigcup_{j\in\mathbb{Z}^+}B_j=\Rn$;\\
$(ii)$~For every $\sigma\geq1$, there exists constants $C,N>0$ such that for any $x\in\Rn$, $\sum_{j\in\mathbb{Z}^+}\chi_{\sigma B_j}(x)\leq C\sigma^N$.
\end{lemma}

\section{Quantitative weighted strong type estimates for variation operators}
This section is concerned with the proofs of Theorem $\ref{quantitative Ap}$ and Theorem $\ref{quantitative two weight}$. Instead of proving Theorem $\ref{quantitative Ap}$ and Theorem $\ref{quantitative two weight}$ directly, we prove more general results, which have its own interests. To be more precise, we consider a non-negative self-adjoint operator $\mathcal{L}$ defined on $L^2(\mathbb{R}^n)$. Thus, $\mathcal{L}$ generates a semigroup $e^{-t\mathcal{L}}$ on $L^2(\mathbb{R}^n)$. Denote the kernel of $e^{-t\mathcal{L}}$ by $p_t(x,y)$, which merely satisfies the following upper bound:
\begin{align}\label{upper bound}
|p_t(x,y)|\leq\frac{C}{t^{n/2}}e^{-\frac{|x-y|^2}{ct}}\Big(1+\frac{\sqrt{t}}{\rho(x)}+
\frac{\sqrt{t}}{\rho(y)}\Big)^{-N},~x,y\in\mathbb{R}^n,~t>0.
\end{align}

We prove the following two theorems. Unlike the case of Schr\"{o}dinger operator, the kernels of heat semigroup do not posses any regularity conditions on spacial variables.
\begin{theorem}\label{general quantitative Ap}
Let $\varrho>2$, $\theta\geq0$ and $\rho$ be a critical radius function. Suppose that $\mathcal{V}_\varrho(e^{-t\mathcal{L}})$ is of weak type $(1,1)$ and bounded on $L^2(\mathbb{R}^n)$. Then for $1<p<\infty$ and $\omega\in A_p^{\rho,\theta}$,
\begin{equation*}
\|\mathcal{V}_\varrho(e^{-t\mathcal{L}})f\|_{L^p(\omega)}
\lesssim[\omega]_{A_p^{\rho,\theta}}^{\max\{1,1/(p-1)\}}\|f\|_{L^p(\omega)}.
\end{equation*}
\end{theorem}

\begin{theorem}\label{general quantitative two weight}
Let $\varrho>2$, $\theta\geq0$, $1<p<\infty$ and $\rho$ be a critical radius function. Suppose that $\mathcal{V}_\varrho(e^{-t\mathcal{L}})$ is of weak type $(1,1)$ and bounded on $L^2(\mathbb{R}^n)$. If $\Phi,\Psi$ are Young functions such that $\bar{\Phi}\in B_{p'}$, $\bar{\Psi}\in B_p$ and a pair $(\mu,v)$ of weights satisfies
\begin{equation*}
[\mu,\nu]_{\Phi,\Psi,p,\rho,\theta}:=\sup_Q\|\nu^{1/p}\|_{\Phi,Q}\|\mu^{-1/p}\|_{\Psi,Q}
\psi_\theta(Q)^{-1}<\infty.
\end{equation*}
Then
\begin{equation*}
\|\mathcal{V}_\varrho(e^{-t\mathcal{L}})f\|_{L^p(\nu)}
\lesssim[\mu,\nu]_{\Phi,\Psi,p,\rho,\theta}[\bar{\Phi}]_{B_{p'}}^{1/p'}[\bar{\Psi}]_{B_p}^{1/p}
\|f\|_{L^p(\mu)}.
\end{equation*}
\end{theorem}
\begin{remark}
Since the kernel of $e^{-t\mathcal{L}}$ only satisfies an upper bound, the tools such as the spatial regularity of the heat kernel and estimates on the difference between the kernels of $e^{-t\mathcal{L}}$ and $e^{t\triangle}$ in \cite{BetFHR,Bui,TZ,WW} are invalid in this new setting.
\end{remark}

The proofs of these two theorems above rely upon the sparse dominations. We need to prove that $\mathcal{M}_{\mathcal{V}_\varrho(e^{-t\mathcal{L}})}$ is of weak type $(1,1)$ in virtue of Proposition \ref{principle}. To do this, we first establish two lemmas.
\begin{lemma}\label{P}
Let $\varrho>2$ and
\begin{equation}\label{PN}
P_{N,t}(\mathcal{L}):=\frac{1}{(N-1)!}\int_t^\infty(s\mathcal{L})^Ne^{-s\mathcal{L}}\frac{ds}{s},~t>0.
\end{equation}
Then for any cube $Q$ and $x,\xi\in Q$,
\begin{align*}
\int_{\mathbb{R}^n}|P_{N,r_Q^2}(\xi,y)|\mathcal{V}_\varrho(e^{-t\mathcal{L}})f(y)dy\lesssim M(\mathcal{V}_\varrho(e^{-t\mathcal{L}})f)(x),
\end{align*}
where $P_{N,t}(x,y)$ is the kernel of $P_{N,t}(\mathcal{L})$.
\end{lemma}
\begin{proof}
From \eqref{PN}, one can check that $P_{N,t}(\mathcal{L})=p(t\mathcal{L})e^{-t\mathcal{L}}$, where $p$ is a polynomial of degree $N-1$ with $p(0)=1$. Then $P_{N,t}(x,y)$ also satisfies \eqref{upper bound}.

Fix a cube $Q\ni x$, using \eqref{upper bound}, we have
\begin{align*}
&\int_{\mathbb{R}^n}|P_{N,r_Q^2}(\xi,y)|\mathcal{V}_\rho(e^{-t\mathcal{L}})f(y)dy\\
&\quad\lesssim\sum_{j=2}^\infty\int_{2^jQ\backslash2^{j-1}Q}r_Q^{-n}e^{-\frac{|\xi-y|^2}{cr_Q^2}}
\mathcal{V}_\varrho(e^{-t\mathcal{L}})f(y)dy+\sum_{k=1}^2\frac{1}{|kQ|}\int_{kQ}
\mathcal{V}_\varrho(e^{-t\mathcal{L}})f(y)dy\\
&\quad\lesssim\sum_{j=2}^\infty\frac{2^{jn}}{|2^jQ|}e^{-c2^{2(j-3)}}
\int_{2^jQ}\mathcal{V}_\varrho(e^{-t\mathcal{L}})f(y)dy
+M(\mathcal{V}_\varrho(e^{-t\mathcal{L}})f)(x)\\
&\quad\lesssim M(\mathcal{V}_\varrho(e^{-t\mathcal{L}})f)(x).
\end{align*}
This proves the conclusion.
\end{proof}

\begin{lemma}\label{estimate for average of PN}
Let $P_{N,t}(\mathcal{L})$ be defined as \eqref{PN}. Then for any function $f$ supported in a ball $B$,
\begin{equation*}
\Big(\frac{1}{|2^jB|}\int_{2^jB\backslash2^{j-1}B}|P_{N,r_B^2}(\mathcal{L})f(x)|^2dx\Big)^{1/2}
\lesssim\alpha(j)\Big(\frac{1}{|B|}\int_B|f(x)|dx\Big),~j\geq3
\end{equation*}
and $\sum_{j=3}^\infty\alpha(j)2^{nj}<\infty$, where $r_B$ is the radius of ball $B$.
\end{lemma}
\begin{proof}
By \eqref{upper bound}, we have
\begin{align*}
&\Big(\frac{1}{|2^jB|}\int_{2^jB\backslash2^{j-1}B}|P_{N,r_B^2}(\mathcal{L})f(x)|^2dx\Big)^{1/2}\\
&\quad\lesssim\Big(\frac{1}{|2^jB|}\int_{2^jB\backslash2^{j-1}B}
\Big|\int_Br_{B}^{-n}e^{-\frac{|x-y|^2}{cr_B^2}}|f(y)|dy\Big|^2dx\Big)^{1/2}\\
&\quad\lesssim\Big(\frac{1}{|2^jB|}\int_{2^jB}
\Big(\frac{1}{|B|}\int_Be^{-c2^{2(j-2)}}|f(y)|dy\Big)^2dx\Big)^{1/2}\\
&\quad=e^{-c2^{2(j-2)}}\Big(\frac{1}{|B|}\int_B|f(x)|dx\Big)=:\alpha(j)\Big(\frac{1}{|B|}\int_B|f(x)|dx\Big).
\end{align*}
From this, it is direct that $\sum_{j=3}^\infty\alpha(j)2^{nj}<\infty$. This completes the proof.
\end{proof}

In the following, we prove that $\mathcal{M}_{\mathcal{V}_\varrho(e^{-t\mathcal{L}})}$ is of weak type $(1,1)$, that is, we have the following proposition.
\begin{proposition}\label{weak type maximal operator}
Let $\varrho>2$. Suppose that $\mathcal{V}_\varrho(e^{-t\mathcal{L}})$ is bounded on $L^2(\mathbb{R}^n)$. Then $\mathcal{M}_{\mathcal{V}_\varrho(e^{-t\mathcal{L}})}$ is of weak type $(1,1)$.
\end{proposition}
\begin{proof}
Fixed $N>0$, applying functional calculus, we get
\begin{equation*}
I=\frac{1}{(N-1)!}\int_0^\infty(s\mathcal{L})^Ne^{-s\mathcal{L}}\frac{ds}{s}.
\end{equation*}
For each $t>0$, we define $S_{N,t}(\mathcal{L})=I-P_{N,t}(\mathcal{L})$, where $P_{N,t}(\mathcal{L})$ is defined as \eqref{PN}.

Let $Q$ be a cube and $x\in Q$. It is obvious that
\begin{align}\label{three terms}
&\|\mathcal{V}_\varrho(e^{-t\mathcal{L}})(f\chi_{\mathbb{R}^n\backslash3Q})\|_{L^\infty(Q)}\\
&\quad\leq\|\mathcal{V}_\varrho(e^{-t\mathcal{L}})P_{N,r_Q^2}(\mathcal{L})
(f\chi_{\mathbb{R}^n\backslash3Q})\|_{L^\infty(Q)}+\|\mathcal{V}_\varrho(e^{-t\mathcal{L}})
S_{N,r_Q^2}(\mathcal{L})(f\chi_{\mathbb{R}^n\backslash3Q})\|_{L^\infty(Q)}\nonumber\\
&\quad\leq\|\mathcal{V}_\varrho(e^{-t\mathcal{L}})P_{N,r_Q^2}(\mathcal{L})
f\|_{L^\infty(Q)}+\|\mathcal{V}_\varrho(e^{-t\mathcal{L}})P_{N,r_Q^2}(\mathcal{L})
(f\chi_{3Q})\|_{L^\infty(Q)}\nonumber\\
&\qquad+\|\mathcal{V}_\varrho(e^{-t\mathcal{L}})S_{N,r_Q^2}(\mathcal{L})
(f\chi_{\mathbb{R}^n\backslash3Q})\|_{L^\infty(Q)}.\nonumber
\end{align}
For brevity, we denote
\begin{equation*}
\mathcal{M}_{\mathcal{V}_\varrho(e^{-t\mathcal{L}})}^1f(x):=\sup_{Q\ni x}
\|\mathcal{V}_\varrho(e^{-t\mathcal{L}})P_{N,r_Q^2}(\mathcal{L})f\|_{L^\infty(Q)},
\end{equation*}
\begin{equation*}
\mathcal{M}_{\mathcal{V}_\varrho(e^{-t\mathcal{L}})}^2f(x):=\sup_{Q\ni x}
\|\mathcal{V}_\varrho(e^{-t\mathcal{L}})P_{N,r_Q^2}(\mathcal{L})(f\chi_{3Q})\|_{L^\infty(Q)},
\end{equation*}
\begin{equation*}
\mathcal{M}_{\mathcal{V}_\varrho(e^{-t\mathcal{L}})}^3f(x):=\sup_{Q\ni x}
\|\mathcal{V}_\varrho(e^{-t\mathcal{L}})S_{N,r_Q^2}(\mathcal{L})
(f\chi_{\mathbb{R}^n\backslash3Q})\|_{L^\infty(Q)}.
\end{equation*}
Then, by \eqref{three terms}, we have
\begin{equation*}
\mathcal{M}_{\mathcal{V}_\varrho(e^{-t\mathcal{L}})}f(x)\leq
\mathcal{M}_{\mathcal{V}_\varrho(e^{-t\mathcal{L}})}^1f(x)+\mathcal{M}_{\mathcal{V}_\varrho(e^{-t\mathcal{L}})}^2f(x)
+\mathcal{M}_{\mathcal{V}_\varrho(e^{-t\mathcal{L}})}^3f(x).
\end{equation*}
Thus, it suffices to prove that $\mathcal{M}_{\mathcal{V}_\varrho(e^{-t\mathcal{L}})}^1f$,
$\mathcal{M}_{\mathcal{V}_\varrho(e^{-t\mathcal{L}})}^2f$ and $\mathcal{M}_{\mathcal{V}_\varrho(e^{-t\mathcal{L}})}^3f$ are of weak type $(1,1)$.

For $\mathcal{M}_{\mathcal{V}_\varrho(e^{-t\mathcal{L}})}^1f$, we first claim that for any $\xi\in Q$,
\begin{equation}\label{claim}
\mathcal{V}_\varrho(e^{-t\mathcal{L}})P_{N,r_Q^2}(\mathcal{L})f(\xi)\leq \int_{\mathbb{R}^n}|P_{N,r_Q^2}(\xi,y)|\mathcal{V}_\rho(e^{-t\mathcal{L}})f(y)dy.
\end{equation}
Indeed, note that
\begin{equation*}
e^{-t\mathcal{L}}P_{N,r_Q^2}(\mathcal{L})=P_{N,r_Q^2}(\mathcal{L})e^{-t\mathcal{L}},~t>0.
\end{equation*}
Then by Minkowski's inequality, we deduce that
\begin{align*}
&\mathcal{V}_\varrho(e^{-t\mathcal{L}})P_{N,r_Q^2}(\mathcal{L})f(\xi)\\
&\quad=\sup_{\{t_i\}\downarrow0}\Big(\sum_{i=1}^\infty|e^{-t_i\mathcal{L}}
P_{N,r_Q^2}(\mathcal{L})(f)(\xi)-e^{-t_{i+1}\mathcal{L}}P_{N,r_Q^2}(\mathcal{L})(f)(\xi)|
^\varrho\Big)^{1/\varrho}\\
&\quad=\sup_{\{t_i\}\downarrow0}\Big(\sum_{i=1}^\infty|P_{N,r_Q^2}(\mathcal{L})e^{-t_i\mathcal{L}}
(f)(\xi)-P_{N,r_Q^2}(\mathcal{L})e^{-t_{i+1}\mathcal{L}}(f)(\xi)|^\varrho\Big)^{1/\varrho}\\
&\quad=\sup_{\{t_i\}\downarrow0}\Big(\sum_{i=1}^\infty\Big|\int_{\mathbb{R}^n}P_{N,r_Q^2}(\xi,y)
\int_{\mathbb{R}^n}(p_{t_i}(z,y)-p_{t_{i+1}}(z,y))f(z)dzdy\Big|^\varrho\Big)^{1/\varrho}\\
&\quad\leq\int_{\mathbb{R}^n}|P_{N,r_Q^2}(\xi,y)|\mathcal{V}_\varrho(e^{-t\mathcal{L}})f(y)dy.
\end{align*}
This proves the claim.

By Lemma \ref{P}, we obtain
\begin{equation*}
\mathcal{M}_{\mathcal{V}_\varrho(e^{-t\mathcal{L}})}^1f\leq M(\mathcal{V}_\varrho(e^{-t\mathcal{L}})f),
\end{equation*}
This, together with the $L^2$-boundedness of $M$ and $\mathcal{V}_\varrho(e^{-t\mathcal{L}})$, implies that $\mathcal{M}_{\mathcal{V}_\varrho(e^{-t\mathcal{L}})}^1$ is bounded on $L^2(\mathbb{R}^n)$. So, in virtue of the criterion in \cite{Au,BlKu}, Lemma \ref{estimate for average of PN} and $I-S_{l,r_B^2}(\mathcal{L})=P_{l,r_B^2}(\mathcal{L})$, the proof of weak type $(1,1)$ of $\mathcal{M}_{\mathcal{V}_\varrho(e^{-t\mathcal{L}})}^1$ is reduced to show that
for every function $f$ supported in a ball $B$ and $j\geq3$,
\begin{equation}\label{criterion}
\Big(\frac{1}{|2^jB|}\int_{2^jB\backslash2^{j-1}B}
\mathcal{M}_{\mathcal{V}_\varrho(e^{-t\mathcal{L}})}^1(S_{l,r_B^2}(\mathcal{L})f)(x)^2dx\Big)
^{1/2}\lesssim2^{-j\alpha}\langle|f|\rangle_B,~\alpha>n,
\end{equation}
where $l$ is a fixed natural number such that $l>\alpha/2+1$.

Bearing in mind that $\supp f\subset B$, for $y\in Q$, we have that
\begin{align*}
|\mathcal{L}^{l+N+1}e^{-(s+u+t)\mathcal{L}}f(y)|&=|(s+u+t)^{-l-N-1}((s+u+t)\mathcal{L})^{l+N+1}
e^{-(s+u+t)\mathcal{L}}f(y)|\\
&\lesssim\int_B\frac{(s+u+t)^{-l-N-1}}{(s+u+t)^{n/2}}e^{-\frac{|y-z|^2}{c(s+u+t)}}|f(z)|dz,
\end{align*}
where we use $(t\mathcal{L})^ke^{-t\mathcal{L}}$ also satisfies \eqref{upper bound} in the last inequality. Note that for $z\in B$, $x\in(2^jB\backslash2^{j-1}B)\cap Q$, $y\in Q$ and $u\geq r_Q^2$,
\begin{equation*}
e^{-\frac{|y-z|^2}{c(s+u+t)}}\lesssim e^{-\frac{2^{2j}r_B^2}{c(s+u+t)}}\lesssim
\frac{(s+u+t)^{\frac{n}{2}+l-1}}{(2^{2j}r_B^2)^{\frac{n}{2}+l-1}}.
\end{equation*}
Therefore,
\begin{align}\label{used for the proof of claim}
&|\mathcal{L}^{l+N+1}e^{-(s+u+t)\mathcal{L}}f(y)|\\
&\quad\lesssim\int_B(s+u+t)^{-l-N-1-\frac{n}{2}}
\frac{(s+u+t)^{\frac{n}{2}+l-1}}{(2^{2j}r_B^2)^{\frac{n}{2}+l-1}}|f(z)|dz\nonumber\\
&\quad\leq(s+u+t)^{-N-2}(2^{2j}r_B^2)^{1-l}\langle|f|\rangle_{B}.\nonumber
\end{align}
Return to the proof of \eqref{criterion}. For $x\in 2^jB\backslash2^{j-1}B$, by \eqref{used for the proof of claim}, it follows that
\begin{align*}
&\mathcal{M}_{\mathcal{V}_\varrho(e^{-t\mathcal{L}})}^1(S_{l,r_B^2}(\mathcal{L})f)(x)\\
&\quad=\sup_{Q\ni x}\esss_{y\in Q}\mathcal{V}_\varrho(e^{-t\mathcal{L}})P_{N,r_Q^2}(\mathcal{L})(S_{l,r_B^2}(\mathcal{L})f)(y)\\
&\quad=\sup_{Q\ni x}\esss_{y\in Q}\sup_{\{t_i\}\downarrow0}\Big(\sum_{i=1}^\infty|e^{-t_i\mathcal{L}}P_{N,r_Q^2}
(S_{l,r_B^2}(\mathcal{L})f)(y)\\
&\qquad-e^{-t_{i+1}\mathcal{L}}P_{N,r_Q^2}
(S_{l,r_B^2}(\mathcal{L})f)(y)|^\varrho\Big)^{1/\varrho}\\
&\quad\leq\sup_{Q\ni x}\esss_{y\in Q}\sup_{\{t_i\}\downarrow0}\sum_{i=1}^\infty\Big|\int_{t_{i+1}}^{t_i}t\mathcal{L}
e^{-t\mathcal{L}}P_{N,r_Q^2}(\mathcal{L})(S_{l,r_B^2}(\mathcal{L})f)(y)\frac{dt}{t}\Big|\\
&\quad\leq\sup_{Q\ni x}\esss_{y\in Q}\int_0^\infty|t\mathcal{L}
e^{-t\mathcal{L}}P_{N,r_Q^2}(\mathcal{L})(S_{l,r_B^2}(\mathcal{L})f)(y)|\frac{dt}{t}\\
&\quad\lesssim\sup_{Q\ni x}\esss_{y\in Q}\int_0^\infty\int_{r_Q^2}^\infty\int_0^{r_B^2}ts^lu^N
|\mathcal{L}^{l+N+1}e^{-(s+u+t)\mathcal{L}}f(y)|\frac{ds}{s}\frac{du}{u}\frac{dt}{t}\\
&\quad\lesssim\langle|f|\rangle_B\sup_{Q\ni x}\int_0^\infty\int_{r_Q^2}^\infty\int_0^{r_B^2}\frac{ts^lu^N}{(s+u+t)^{N+2}}
(2^{2j}r_B^2)^{1-l}\frac{ds}{s}\frac{du}{u}\frac{dt}{t}\\
&\quad\leq\langle|f|\rangle_B\sup_{Q\ni x}\int_0^\infty\int_{r_Q^2}^\infty\int_0^{r_B^2}s^{l-1}
(2^jr_B)^{2-2l}\frac{u^{N-1}}{(s+u)^{N+1/2}}\frac{1}{(s+t)^{3/2}}dsdudt.
\end{align*}
A simple computation shows that
\begin{align*}
\int_0^\infty\frac{1}{(s+t)^{3/2}}dt\sim s^{-1/2},
\end{align*}
\begin{align*}
\int_{r_Q^2}^\infty\frac{u^{N-1}}{(s+u)^{N+\frac{1}{2}}}du\leq
\int_0^\infty\frac{(s+u)^{N-1}}{(s+u)^{N+\frac{1}{2}}}du\sim s^{-1/2}
\end{align*}
and
\begin{align*}
\int_0^{r_B^2}s^{l-2}(2^jr_B)^{2-2l}ds\sim2^{-j(2l-2)}.
\end{align*}
Hence,
\begin{align*}
\mathcal{M}_{\mathcal{V}_\varrho(e^{-t\mathcal{L}})}^1(S_{l,r_B^2}(\mathcal{L})f)(x)
\lesssim2^{-j(2l-2)}\langle|f|\rangle_B\leq2^{-j\alpha}\langle|f|\rangle_B.
\end{align*}
This proves \eqref{criterion}, and then $\mathcal{M}_{\mathcal{V}_\varrho(e^{-t\mathcal{L}})}^1$ is of weak type $(1,1)$.

Next, we consider $\mathcal{M}_{\mathcal{V}_\varrho(e^{-t\mathcal{L}})}^2$. Note that
\begin{align*}
|\mathcal{L}^{k+1}e^{-(s+t)\mathcal{L}}(f\chi_{3Q})(y)|&=
(s+t)^{-k-1}|((s+t)\mathcal{L})^{k+1}e^{-(s+t)\mathcal{L}}(f\chi_{3Q})(y)|\\
&\lesssim(s+t)^{-k-1}\int_{3Q}(s+t)^{-n/2}|f(z)|dz.\nonumber
\end{align*}
Then we have
\begin{align*}
&\|\mathcal{V}_\varrho(e^{-t\mathcal{L}})P_{N,r_Q^2}(f\chi_{3Q})\|_{L^\infty(Q)}\\
&\quad=\esss_{y\in Q}\sup_{\{t_j\}\downarrow0}\sum_{i=1}^\infty\Big|\int_{t_{i+1}}^{t_i}t\mathcal{L}
e^{-t\mathcal{L}}P_{N,r_Q^2}(f\chi_{3Q})(y)\frac{dt}{t}\Big|\\
&\quad\leq\esss_{y\in Q}\int_0^\infty|t\mathcal{L}
e^{-t\mathcal{L}}P_{N,r_Q^2}(f\chi_{3Q})(y)|\frac{dt}{t}\\
&\quad\lesssim\sum_{k=0}^{N-1}\esss_{y\in Q}\int_0^\infty\int_{r_Q^2}^\infty\frac{ts^k}{(s+t)^{k+1}}|((s+t)\mathcal{L})^{k+1}
e^{-(s+t)\mathcal{L}}(f\chi_{3Q})(y)|\frac{ds}{s}\frac{dt}{t}\\
&\quad\lesssim\sum_{k=0}^{N-1}\int_0^\infty\int_{r_Q^2}^\infty\frac{ts^k}{(s+t)^{k+1}}
\int_{3Q}\frac{1}{(s+t)^{n/2}}|f(z)|dz\frac{ds}{s}\frac{dt}{t}\\
&\quad=\sum_{k=0}^{N-1}\int_0^\infty\int_{r_Q^2}^\infty\int_{3Q}\frac{ts^k}{(s+t)^{k+\frac{n}{2}+1}}
|f(z)|dz\frac{ds}{s}\frac{dt}{t}\\
&\quad\sim\sum_{k=0}^{N-1}\int_0^\infty\int_{r_Q^2}^\infty\frac{ts^kr_Q^n}
{(s+t)^{k+\frac{n}{2}+1}}\frac{ds}{s}\frac{dt}{t}\langle|f|\rangle_{3Q}\lesssim Mf(x),
\end{align*}
where we use the fact that
\begin{equation*}
\int_0^\infty\frac{t}{(s+t)^{k+\frac{n}{2}+1}}\frac{dt}{t}\sim s^{-k-\frac{n}{2}},~
\int_{r_Q^2}^\infty r_Q^ns^{-1-\frac{n}{2}}ds\sim1
\end{equation*}
in the last inequality. Thus, the weak type $(1,1)$ of $\mathcal{M}_{\mathcal{V}_\varrho(e^{-t\mathcal{L}})}^2$ follows from the weak type $(1,1)$ of $M$.

Finally, we deal with $\mathcal{M}_{\mathcal{V}_\varrho(e^{-t\mathcal{L}})}^3$. Since $e^{-x}\lesssim x^{-\beta}$ for any $\beta>0$, for $y\in Q$, we have
\begin{align*}
&|\mathcal{L}^{N+1}e^{-(s+t)\mathcal{L}}(f\chi_{\mathbb{R}^n\backslash3Q})(y)|\\
&\quad\lesssim(s+t)^{-N-1}\int_{\mathbb{R}^n}(s+t)^{-n/2}e^{-\frac{|z-y|^2}{c(s+t)}}|f(z)|
\chi_{\mathbb{R}^n\backslash2Q}(z)dz\\
&\quad\lesssim\sum_{k=2}^\infty(s+t)^{-N-1}\int_{2^kQ\backslash2^{k-1}Q}(s+t)^{-n/2}
\frac{(s+t)^{\frac{n}{2}+N-1}}{|z-y|^{2(\frac{n}{2}+N-1)}}|f(z)|dz\\
&\quad\lesssim\sum_{k=2}^\infty\int_{2^kQ}
\frac{1}{(s+t)^2(2^kr_Q)^{n+2N-2}}|f(z)|dz.
\end{align*}
From this, we deduce that
\begin{align*}
&\|\mathcal{V}_\varrho(e^{-t\mathcal{L}})S_{N,r_Q^2}(\mathcal{L})
(f\chi_{\mathbb{R}^n\backslash3Q})\|_{L^\infty(Q)}\\
&\quad=\esss_{y\in Q}\sup_{\{t_i\}\downarrow0}\Big(\sum_{i=1}^\infty|e^{-t_i\mathcal{L}}S_{N,r_Q^2}(\mathcal{L})
(f\chi_{\mathbb{R}^n\backslash3Q})(y)\\
&\qquad-e^{-t_{i+1}\mathcal{L}}S_{N,r_Q^2}(\mathcal{L})
(f\chi_{\mathbb{R}^n\backslash3Q})(y)|^\varrho\Big)^{1/\varrho}\\
&\quad\leq\esss_{y\in Q}\int_0^\infty|t\mathcal{L}e^{-t\mathcal{L}}S_{N,r_Q^2}(\mathcal{L})
(f\chi_{\mathbb{R}^n\backslash3Q})(y)|\frac{dt}{t}\\
&\quad\lesssim\esss_{y\in Q}\int_0^\infty\int_0^{r_Q^2}ts^N|\mathcal{L}^{N+1}e^{-(s+t)\mathcal{L}}
(f\chi_{\mathbb{R}^n\backslash3Q})(y)|\frac{ds}{s}\frac{dt}{t}\\
&\quad\lesssim\sum_{k=2}^\infty\int_0^\infty\int_0^{r_Q^2}\int_{2^kQ}
\frac{s^{N-1}}{(s+t)^{2}(2^kr_Q)^{n+2N-2}}|f(z)|dzdsdt.
\end{align*}
It is easy to see that
\begin{equation*}
\int_0^{r_Q^2}(2^kr_Q)^{-2N+2}s^{N-2}ds\sim2^{-2(N-1)k},~\int_0^\infty(s+t)^{-2}dt= s^{-1}.
\end{equation*}
Then we further have
\begin{equation*}
\|\mathcal{V}_\varrho(e^{-t\mathcal{L}})S_{N,r_Q^2}(\mathcal{L})
(f\chi_{\mathbb{R}^n3Q})\|_{L^\infty(Q)}\lesssim\sum_{k=2}^\infty2^{-2(N-1)k}\langle|f|\rangle_{2^kQ}
\lesssim Mf(x).
\end{equation*}
It follows that $\mathcal{M}_{\mathcal{V}_\varrho(e^{-t\mathcal{L}})}^3$ is of weak type $(1,1)$. This completes the proof.
\end{proof}
\begin{proof}[Proof of Theorem \ref{general quantitative Ap}]
Let $\{B_j:=B(x_j,\rho(x_j))\}$ be a family of critical balls in Lemma \ref{covering lemma}. Then for each $B_j$, there exists $Q_j\in\mathcal{D}$ such that $B_j\subset Q_j\subset\alpha_0B_j$ (see \cite{HyKa}). Hence, we also have $\bigcup_jQ_j=\mathbb{R}^n$ and $\sum_j\chi_{Q_j}\leq C$ for some $C>0$. From this, we write
\begin{align}\label{I1+I2}
\|\mathcal{V}_\varrho(e^{-t\mathcal{L}})f\|_{L^p(\omega)}^p&\leq\sum_j\int_{B_j}
\mathcal{V}_\varrho(e^{-t\mathcal{L}})(f)(x)^p\omega(x)dx\\
&\lesssim\sum_j\int_{B_j}\mathcal{V}_\varrho(e^{-t\mathcal{L}})
(f\chi_{\mathbb{R}^n\backslash3Q_j})(x)^p\omega(x)dx
\nonumber\\
&\quad+\sum_j\int_{B_j}\mathcal{V}_\varrho(e^{-t\mathcal{L}})(f\chi_{3Q_j})(x)^p\omega(x)dx\nonumber\\
&=:I_1+I_2.\nonumber
\end{align}

We first deal with the term $I_1$. By noting that $\rho(x)\sim\rho(x_j)$ for $x\in B_j$, the kernel of $t\mathcal{L}e^{-t\mathcal{L}}$ also satisfies \eqref{upper bound}, we have
\begin{align*}
\mathcal{V}_\varrho(e^{-t\mathcal{L}})(f\chi_{\mathbb{R}^n\backslash3Q_j})(x)&\leq
\sup_{\{t_k\}\downarrow0}\sum_k\Big|\int_{t_{k+1}}^{t_k}t\mathcal{L}e^{-t\mathcal{L}}
(f\chi_{\mathbb{R}^n\backslash3Q_j})(x)\frac{dt}{t}\Big|\\
&\leq\int_0^\infty|t\mathcal{L}e^{-t\mathcal{L}}
(f\chi_{\mathbb{R}^n\backslash3Q_j})(x)|\frac{dt}{t}\\
&\lesssim\int_0^\infty\int_{\mathbb{R}^n\backslash 3 B_j}t^{-n/2}\Big(1+\frac{\sqrt{t}}{\rho(x_j)}\Big)^{-\sigma-1}e^{-\frac{|x-y|^2}{ct}}|f(y)|
dy\frac{dt}{t},
\end{align*}
where $\sigma$ is determined later. Since $|y-x_j|>\rho(x_j)$ for $y\in\mathbb{R}^n\backslash 3 B_j$, we divide $\mathcal{V}_\varrho(e^{-t\mathcal{L}})(f\chi_{\mathbb{R}^n\backslash3Q_j})$ into the following three terms:
\begin{align}\label{global}
&\mathcal{V}_\varrho(e^{-t\mathcal{L}})(f\chi_{\mathbb{R}^n\backslash3Q_j})(x)\\
&\quad\lesssim\int_0^{\rho(x_j)^2}\int_{\mathbb{R}^n\backslash 3 B_j}t^{-n/2}\Big(1+\frac{\sqrt{t}}{\rho(x_j)}\Big)^{-\sigma-1}e^{-\frac{|x-y|^2}{ct}}|f(y)|dy\frac{dt}{t}
\nonumber\\
&\qquad+\int_{\rho(x_j)^2}^{|y-x_j|^2}\int_{\mathbb{R}^n\backslash 3 B_j}t^{-n/2}\Big(1+\frac{\sqrt{t}}{\rho(x_j)}\Big)^{-\sigma-1}e^{-\frac{|x-y|^2}{ct}}|f(y)|dy\frac{dt}{t}
\nonumber\\
&\qquad+\int_{|y-x_j|^2}^\infty\int_{\mathbb{R}^n\backslash 3 B_j}t^{-n/2}\Big(1+\frac{\sqrt{t}}{\rho(x_j)}\Big)^{-\sigma-1}e^{-\frac{|x-y|^2}{ct}}|f(y)|dy\frac{dt}{t}\nonumber\\
&\quad=:I_{11}+I_{12}+I_{13}.\nonumber
\end{align}
Before we estimate $I_{11}$, $I_{12}$, $I_{13}$. We prove an estimate that is frequently used. For any $\sigma>0$,
\begin{align}\label{frequently used}
&\sum_{k=1}^\infty\frac{1}{|3^{k+1}B_j|}\int_{3^{k+1}B_j}|f(y)|
\Big(\frac{\rho(x_j)}{3^k\rho(x_j)}\Big)^{\sigma+1}dy\\
&\quad\lesssim\sum_{k=1}^\infty3^{-k}\Big(1+\frac{3^{k+1}\rho(x_j)}{\rho(x_j)}\Big)^{-\sigma}
\frac{1}{|3^{k+1}B_j|}\int_{3^{k+1}B_j}|f(y)|dy\lesssim M^{\rho,\sigma}f(x).\nonumber
\end{align}

We consider $I_{11}$ first. For any $\sigma>0$ and $x\in B_j$, by \eqref{frequently used},
\begin{align}\label{I1}
I_1&\lesssim\int_0^{\rho(x_j)^2}\int_{\mathbb{R}^n\backslash 3 B_j}t^{-n/2}\frac{|f(y)|}{|x-y|^{n+\sigma+1}}t^{\frac{n+\sigma+1}{2}}dy\frac{dt}{t}\\
&=\int_0^{\rho(x_j)^2}\int_{\mathbb{R}^n\backslash 3 B_j}\frac{|f(y)|}{|x-y|^n}
\Big(\frac{\sqrt{1}}{|x-y|}\Big)^{\sigma+1}dyt^{\frac{\sigma+1}{2}-1}dt\nonumber\\
&\lesssim\sum_{k=1}^\infty\int_{3^{k+1}B_j\backslash3^kB_j}\frac{|f(y)|}{|x-y|^n}
\Big(\frac{\rho(x_j)}{|x-y|}\Big)^{\sigma+1}dy\nonumber\\
&\lesssim\sum_{k=1}^\infty\frac{1}{|3^{k+1}B_j|}\int_{3^{k+1}B_j}|f(y)|
\Big(\frac{\rho(x_j)}{3^k\rho(x_j)}\Big)^{\sigma+1}dy\nonumber\\
&\lesssim M^{\rho,\sigma}f(x).\nonumber
\end{align}

While for $I_{12}$. For any $\sigma>0$ and $x\in B_j$, again by \eqref{frequently used},
\begin{align}\label{I2}
I_{12}&\lesssim\int_{\rho(x_j)^2}^{|x_j-y|^2}\int_{\mathbb{R}^n\backslash3B_j}
\frac{|f(y)|}{t^{n/2}}\Big(\frac{\rho(x_j)}{\sqrt{t}}\Big)^{\sigma+1}\frac{t^{\frac{\sigma+n+2}{2}}}
{|x-y|^{\sigma+n+2}}
dy\frac{dt}{t}\\
&\leq\int_0^{|x_j-y|^2}\int_{\mathbb{R}^n\backslash3B_j}
\frac{|f(y)|}{|x-y|^n}\Big(\frac{\rho(x_j)}{\sqrt{t}}\Big)^{\sigma+1}
\Big(\frac{\sqrt{t}}{|x-y|}\Big)^{\sigma+2}dy\frac{dt}{t}\nonumber\\
&\sim\int_{\mathbb{R}^n\backslash3B_j}\frac{|f(y)|}{|x-y|^n}
\Big(\frac{\rho(x_j)}{|x-y|}\Big)^{\sigma+1}dy\nonumber\\
&\lesssim\sum_{k=1}^\infty\frac{1}{|3^{k+1}B_j|}\int_{3^{k+1}B_j}|f(y)|
\Big(\frac{\rho(x_j)}{3^k\rho(x_j)}\Big)^{\sigma+1}dy\nonumber\\
&\lesssim M^{\rho,\sigma}f(x).\nonumber
\end{align}

Finally, we deal with $I_{13}$. For any $\sigma>0$ and $x\in B_j$, by \eqref{frequently used}, we find that
\begin{align}\label{I3}
I_{13}&\lesssim\int_{|x_j-y|^2}^{\infty}\int_{\mathbb{R}^n\backslash3B_j}
\frac{|f(y)|}{|x_j-y|^n}\Big(\frac{\rho(x_j)}{\sqrt{t}}\Big)^{\sigma+1}dy\frac{dt}{t}\\
&\sim\int_{\mathbb{R}^n\backslash3B_j}
\frac{|f(y)|}{|x_j-y|^n}\Big(\frac{\rho(x_j)}{|x_j-y|}\Big)^{\sigma+1}dy\nonumber\\
&\lesssim\sum_{k=1}^\infty\frac{1}{|3^{k+1}B_j|}\int_{3^{k+1}B_j}|f(y)|
\Big(\frac{\rho(x_j)}{3^k\rho(x_j)}\Big)^{\sigma+1}dy\nonumber\\
&\lesssim M^{\rho,\sigma}f(x).\nonumber
\end{align}

Combining with \eqref{global},\eqref{I1}-\eqref{I3}, for $x\in B_j$, we obtain
\begin{equation}\label{global estimate}
\mathcal{V}_\varrho(e^{-t\mathcal{L}})(f\chi_{\mathbb{R}^n\backslash3Q_j})(x)\lesssim
M^{\rho,\sigma}f(x).
\end{equation}
By choosing $\sigma>\theta/(p-1)$ and Lemmas \ref{quantitative weighted bounds for Hardy-Littlewood maximal operators} and \ref{covering lemma}, it follows that
\begin{align}\label{I}
I_1&\lesssim\sum_j\int_{B_j}M^{\rho,\sigma}f(x)^p\omega(x)dx\\
&\lesssim\int_{\mathbb{R}^n}M^{\rho,\sigma}f(x)^p\omega(x)dx\nonumber\\
&\lesssim[\omega]_{A_{p}^{\rho,\theta}}^{p\max\{1,\frac{1}{p-1}\}}\|f\|_{L^p(\omega)}^p.\nonumber
\end{align}

Next, we take care of the term $I_2$. In virtue of Proposition \ref{principle} and Proposition \ref{weak type maximal operator}, there exists a sparse family $\mathcal{S}_j\subseteq Q_j$ for each $j$ such that
\begin{equation*}
\mathcal{V}_\varrho(e^{-t\mathcal{L}})(f\chi_{3Q_j})(x)\lesssim\sum_{P\in S_j}\langle|f|\rangle_{3P}\chi_P(x),~x\in Q_j.
\end{equation*}
Observe that $P\subseteq Q_j$, the definition of critical radius function yields that
\begin{equation*}
r_P\leq r_{Q_j}\sim\rho(x_j)\sim\rho(x_P),
\end{equation*}
where $r_P$, $x_P$ are the side-length and center of the cube $P$, respectively. Then for any $N>0$ and $x\in Q_j$,
\begin{align*}
\sum_{P\in\mathcal{S}_j}\langle|f|\rangle_{3P}\chi_P(x)=
\sum_{P\in\mathcal{S}_j}\langle|f|\rangle_{3P}\Big(1+\frac{r_p}{\rho(x_p)}\Big)^{-N}
\Big(1+\frac{r_p}{\rho(x_p)}\Big)^{N}\chi_P(x)
\lesssim\mathcal{A}_{\mathcal{S}_j}^{\rho,N}f(x),
\end{align*}
which further implies that
\begin{equation}\label{local estimate}
\mathcal{V}_\varrho(e^{-t\mathcal{L}})(f\chi_{3Q_j})(x)
\lesssim\mathcal{A}_{\mathcal{S}_j}^{\rho,N}(f\chi_{3Q_j})(x),~x\in Q_j.
\end{equation}
By taking $N\geq\theta\max\{1,\frac{1}{p-1}\}$ and using Lemmas \ref{covering lemma}, \ref{weighted estimates for sparse operator}, we get that
\begin{align}\label{II}
I_2&\lesssim\sum_j\int_{\mathbb{R}^n}\mathcal{A}_{\mathcal{S}_j}^{\rho,N}(f\chi_{3Q_j})(x)^p
\omega(x)dx\\
&\lesssim[\omega]_{A_p^{\rho,\theta}}^{p\max\{1,\frac{1}{p-1}\}}\sum_j\int_{\mathbb{R}^n}
|f(x)|^p\chi_{3Q_j}(x)\omega(x)dx\nonumber\\
&\lesssim[\omega]_{A_p^{\rho,\theta}}^{p\max\{1,\frac{1}{p-1}\}}\int_{\mathbb{R}^n}
|f(x)|^p\omega(x)dx.\nonumber
\end{align}

Hence, by \eqref{I1+I2}, \eqref{I} and \eqref{II}, we have that
\begin{equation*}
\|\mathcal{V}_\varrho(e^{-t\mathcal{L}})f\|_{L^p(\omega)}\lesssim[\omega]_{A_p^{\rho,\theta}}^{\max\{1,\frac{1}{p-1}\}}
\|f\|_{L^p(\omega)}.
\end{equation*}
This is the expected result.
\end{proof}

Next, we prove Theorem \ref{general quantitative two weight}. The two-weight inequality for $M^{\rho,\sigma}$ will be used.

\begin{proof}[Proof of Theorem \ref{general quantitative two weight}]
We still use the notations as in the proof of Theorem \ref{quantitative Ap}. It is trivial that
\begin{align*}
\|\mathcal{V}_\varrho(e^{-t\mathcal{L}})f\|_{L^p(\nu)}^p&\lesssim\sum_j\int_{B_j}
\mathcal{V}_\varrho(e^{-t\mathcal{L}})(f\chi_{\mathbb{R}^n\backslash3Q_j})(x)^p\nu(x)dx\\
&\quad+\sum_j\int_{B_j}
\mathcal{V}_\varrho(e^{-t\mathcal{L}})(f\chi_{3Q_j})(x)^p\nu(x)dx\\
&=:J+K.
\end{align*}

In virtue of \eqref{global estimate}, Lemmas \ref{quantitative weighted bounds for Hardy-Littlewood maximal operators} and \ref{covering lemma}, choose $\sigma$ large enough, we have
\begin{align*}
J&\lesssim\sum_j\int_{B_j}M^{\rho,\sigma}f(x)^p\nu(x)dx\\
&\lesssim\int_{\mathbb{R}^n}M^{\rho,\sigma}f(x)^p\nu(x)dx\\
&\lesssim[\mu,\nu]_{\Phi,\Psi,p,\rho,\theta}^p[\bar{\Psi}]_{B_p}\|f\|_{L^p(\mu)}^p,
\end{align*}
which implies that
$$J^{1/p}\lesssim[\mu,\nu]_{\Phi,\Psi,p,\rho,\theta}[\bar{\Psi}]_{B_p}^{1/p}\|f\|_{L^p(\mu)}.$$

Next, we consider $K$. We use \eqref{global estimate} and duality to get that
\begin{align*}
K&\lesssim\sum_j\int_{\mathbb{R}^n}\mathcal{A}_{\mathcal{S}_j}^{\rho,\theta}(f\chi_{3Q_j})(x)^p\nu(x)dx\\
&=\sum_j\left(\underset{g\geq0}{\sup_{\|g\|_{L^{p'}(\nu)\leq1}}}\int_{\mathbb{R}^n}
\mathcal{A}_{\mathcal{S}_j}^{\rho,\theta}(f\chi_{3Q_j})(x)g(x)\nu(x)dx\right)^p.
\end{align*}
Since $\mathcal{S}_j$ is a sparse family, by making use of generalized H\"{o}lder's inequality, the bump condition, Lemma \ref{boundedness of Orlicz maximal operator}, we deduce that
\begin{align*}
&\int_{\mathbb{R}^n}
\mathcal{A}_{\mathcal{S}_j}^{\rho,\theta}(f\chi_{3Q_j})(x)g(x)\nu(x)dx\\
&\quad\lesssim\sum_{Q\in\mathcal{S}_j}\Big(\frac{1}{|3Q|}\int_{3Q}|f(x)\chi_{3Q_j}|dx\Big)
\Big(1+\frac{r_{3Q}}{\rho(x_Q)}\Big)^{-\theta}\int_Qg(x)\nu(x)dx\\
&\quad\lesssim\|f\chi_{3Q_j}\mu^{1/p}\|_{\bar{\Psi},3Q}\|\mu^{-1/p}\|_{\Psi,3Q}
\|g\nu^{1/p'}\|_{\bar{\Phi},3Q}\|\nu^{1/p}\|_{\Phi,3Q}|E_Q|\Big(1+\frac{r_{3Q}}{\rho(x_Q)}\Big)^{-\theta}\\
&\quad\leq[\mu,\nu]_{\Phi,\Psi,p,\rho,\theta}\int_{\mathbb{R}^n}M_{\bar{\Psi}}
(f\chi_{3Q_j}\mu^{1/p})(x)M_{\bar{\Phi}}(g\nu^{1/p'})(x)dx\\
&\quad\leq[\mu,\nu]_{\Phi,\Psi,p,\rho,\theta}\|M_{\bar{\Psi}}(f\chi_{3Q_j}\mu^{1/p})\|_{L^p}
\|M_{\bar{\Phi}}(g\nu^{1/p'})\|_{L^{p'}}\\
&\quad\lesssim[\mu,\nu]_{\Phi,\Psi,p,\rho,\theta}[\bar{\Phi}]_{B_{p'}}^{1/p'}[\bar{\Psi}]_{B_p}^{1/p}
\|f\chi_{3Q_j}\|_{L^{p}(\mu)}\|g\|_{L^{p'}(\nu)}.
\end{align*}
Hence, utilize Lemma \ref{covering lemma}, we have
\begin{align*}
K^{1/p}&\lesssim\Big(\sum_j[\mu,\nu]_{\Phi,\Psi,p,\rho,\theta}^p[\bar{\Phi}]_{B_{p'}}^{p/p'}
[\bar{\Psi}]_{B_p}\int_{\mathbb{R}^n}|f(x)|^p\chi_{3Q_j}(x)\mu(x)dx\Big)^{1/p}\\
&\lesssim[\mu,\nu]_{\Phi,\Psi,p,\rho,\theta}[\bar{\Phi}]_{B_{p'}}^{1/p'}[\bar{\Psi}]_{B_p}^{1/p}
\|f\|_{L^{p}(\mu)}.
\end{align*}
Combining the estimates of $J^{1/p}$ and $K^{1/p}$, we prove the result.
\end{proof}

\begin{proof}[Proofs of Theorems \ref{quantitative Ap} and \ref{quantitative two weight}]
When $\mathcal{L}$ is a Schr\"{o}dinger operator, it is easy to check that the kernel of $e^{-t\mathcal{L}}$ satisfies \eqref{upper bound}. Moreover, we know from \cite{BetFHR} that $\mathcal{V}_\varrho(e^{-t\mathcal{L}})$ is of weak type $(1,1)$ and bounded on $L^2(\mathbb{R}^n)$. Therefore, Theorems \ref{quantitative Ap} and \ref{quantitative two weight} directly follow from Theorems \ref{general quantitative Ap} and \ref{general quantitative two weight}, respectively.
\end{proof}

\section{Mixed weak type inequality for variation operators}
This section is going to prove Theorem \ref{mixed weak type inequality for variation of heat semigroup in Sch}. We first establish weighted mixed weak type inequality for variation operators associated with classical heat semigroups.
\begin{lemma}\label{mixed weak type inequality for variation operators associated with classical heat semigroup}
Let $\varrho>2$, $\mu\in A_{1}$ and $\nu\in A_\infty$. Then for any $\lambda>0$, there holds
\begin{align*}
\lambda\mu\nu\Big(\Big\{x\in\mathbb{R}^n:\frac{\mathcal{V}_\varrho(e^{-t\Delta})(f\nu)(x)}{\nu(x)}>
\lambda\Big\}\Big)\lesssim\int_{\mathbb{R}^n}|f(x)|\mu(x)\nu(x)dx.
\end{align*}
\end{lemma}
\begin{proof}
For each $f\in L_c^\infty(\mathbb{R}^n)$, it was pointed out in \cite{GuWWY} that there exist $3^n$
dyadic lattices $D^j$ and $1/(2\cdot9^n)$-sparse families $S_j\subseteq D^j$ such that for a.e. $x\in\mathbb{R}^n$,
\begin{align*}
\mathcal{V}_\varrho(e^{-t\Delta})(f)(x)\lesssim\sum_{j=1}^{3^n}\sum_{Q\in S_j}\Big(\frac{1}{|Q|}\int_Q|f(x)|dx\Big)\chi_Q(x).
\end{align*}
Thus, following a similar scheme as in \cite{WWZ}, for $0<p<\infty$ and $\omega\in A_\infty$, one can prove the following Coifman-Fefferman inequality,
\begin{align*}
\|\mathcal{V}_\varrho(e^{-t\Delta})(f)\|_{L^p(\omega)}\lesssim\|Mf\|_{L^p(\omega)},
\end{align*}
this combined with the standard methods in \cite{LiOP} leads to the desired result.
\end{proof}

Relied upon Lemma \ref{mixed weak type inequality for variation operators associated with classical heat semigroup}, we have the following lemma.
\begin{lemma}\label{mixed weak type inequality for local}
Let $\varrho>2$, $n\geq3$, $\rho$ be a critical radius function, $B_x:=B(x,\rho(x))$ with $x\in\mathbb{R}^n$. Let $\mu\in A_{1}^{\rho,loc}$, $\nu\in A_\infty^{\rho,loc}$. Then for any $\lambda>0$, there holds
\begin{align*}
\lambda\mu\nu\Big(\Big\{x\in\mathbb{R}^n:\frac{\mathcal{V}_\varrho(e^{-t\Delta})(f\chi_{B_x}\nu)(x)}{\nu(x)}>
\lambda\Big\}\Big)\lesssim\int_{\mathbb{R}^n}|f(x)|\mu(x)\nu(x)dx.
\end{align*}
\end{lemma}
\begin{proof}
Let $\{B_j:B(x_j,\rho(x_j))\}_{j\in \mathbb{Z}^+}$ be the family of balls given in Lemma \ref{covering lemma}. Let $\widetilde{B_j}:=\tau_0 B_j$, where $\tau_0=1+C_02^{\frac{N_0}{N_0+1}}$ with $N_0$ and $C_0$ are given as in \eqref{critical radius function}. In the following, we claim that for any $\mu\in A_1^{\rho,loc}$ and $\nu\in A_\infty^{\rho,loc}$, there exists a $s>1$ such that
$$\mu|_{\widetilde{B_j}}\in A_1(\widetilde{B_j}),~\nu|_{\widetilde{B_j}}\in A_s(\widetilde{B_j}).$$
For any ball $B:=B(x_B,r_B)\subset\widetilde{B_j}$, we prove it by dividing it into two cases:

{\bf Case 1. $r_B\leq\tau_0\rho(x_B)$:}\, Using $[\mu]_{A_1^{\rho,loc}}\sim[\mu]_{A_1^{\tau_0\rho,loc}}$, we have
\begin{align*}
\Big(\frac{1}{|B|}\int_B\mu(y)dy\Big)\Big(\inf_{B\ni y}\mu(y)\Big)^{-1}\leq[\mu]_{A_1^{\tau_0\rho,loc}}\sim[\mu]_{A_1^{\rho,loc}}.
\end{align*}

{\bf Case 2. $r_B>\tau_0\rho(x_B)$:}\, In this case, $B(x_B,\tau_0\rho(x_B))\subset B\subset\widetilde{B_j}$. Since $|x_j-x_B|<\tau_0\rho(x_j)$, $\rho(x_B)\sim\rho(x_j)$, which yields that $|B|\sim|\widetilde{B_j}|$. Therefore,
\begin{align*}
\Big(\frac{1}{|B|}\int_B\mu(y)dy\Big)\Big(\inf_{B\ni y}\mu(y)\Big)^{-1}&\lesssim\Big(\frac{1}{|\widetilde{B_j}|}\int_{\widetilde{B_j}}\mu(y)dy\Big)
\Big(\inf_{\widetilde{B_j}\ni y}\mu(y)\Big)^{-1}\\
&\leq[\mu]_{A_1^{\tau_0\rho,loc}}\sim[\mu]_{A_1^{\rho,loc}}.
\end{align*}
Combined with Case 1 and Case 2, we conclude that $\mu|_{\widetilde{B_j}}\in A_1(\widetilde{B_j})$. While for $\nu\in A_{\infty}^{\rho,loc}$, there is a $s>1$ such that $\nu\in A_{s}^{\rho,loc}$. Similarly to the proof of $\mu|_{\widetilde{B_j}}\in A_1(\widetilde{B_j})$, we can also prove that $\nu|_{\widetilde{B_j}}\in A_s(\widetilde{B_j})$. This proves the claim.

Next, we return to the proof of lemma. By Lemma \ref{extension}, $\mu|_{\widetilde{B_j}}$ and $\nu|_{\widetilde{B_j}}$ admit extensions $\mu_j$ and $\nu_j$ on $\mathbb{R}^n$, respectively, which satisfy $\mu_j\in A_1$ and $\nu_j\in A_s\subseteq A_\infty$. In addition, for any $j\in \mathbb{Z}^+$, $x\in B_j$ and $y\in B_x$, observe that
\begin{align*}
|y-x_j|&\leq|y-x|+|x-x_j|\leq\rho(x)+\rho(x_j)\\
&\leq C_0\rho(x_j)\Big(1+\frac{\rho(x_j)}{\rho(x_j)}\Big)^{\frac{N_0}{N_0+1}}+\rho(x_j)=\tau_0\rho(x_j),
\end{align*}
which implies that $B_x\subseteq\widetilde{B_j}$ for any $x\in B_j$ and $j\in\mathbb{Z}^+$. Then, by Lemma \ref{covering lemma} and Lemma \ref{mixed weak type inequality for variation operators associated with classical heat semigroup}, we deduce that
\begin{align*}
&\lambda\mu\nu\Big(\Big\{x\in\mathbb{R}^n:\frac{\mathcal{V}_\varrho(e^{-t\Delta})(f\chi_{B_x}\nu)(x)}
{\nu(x)}>\lambda\Big\}\Big)\\
&\quad\leq\lambda\sum_j\mu_j\nu_j\Big(\Big\{x\in B_j:\frac{\mathcal{V}_\varrho(e^{-t\Delta})(f\chi_{B_x}\nu)(x)}
{\nu(x)}>\lambda\Big\}\Big)\\
&\quad\leq\lambda\sum_j\mu_j\nu_j\Big(\Big\{x\in\mathbb{R}^n:\frac{\mathcal{V}_\varrho(e^{-t\Delta})
(f\chi_{B_x}\nu)(x)}{\nu(x)}>\lambda\Big\}\Big)\\
&\quad\lesssim\sum_j\int_{B_x}|f(x)|\mu_j(x)\nu_j(x)dx
\leq\sum_j\int_{\widetilde{B_j}}|f(x)|\mu(x)\nu(x)dx\\
&\quad\lesssim\int_{\mathbb{R}^n}|f(x)|\mu(x)\nu(x)dx.
\end{align*}
\end{proof}

The next lemma tells us that the conditions of weights in Theorem \ref{mixed weak type inequality for variation of heat semigroup in Sch} are stronger than ones in Lemma \ref{mixed weak type inequality for local}.
\begin{lemma}\label{strong condition implies weak condition}
If $\mu\in A_{1}^{\rho}$, $\nu\in A_\infty^{\rho}(\mu)$, then $\mu\in A_{1}^{\rho,loc}$, $\nu\in A_\infty^{\rho,loc}$.
\end{lemma}
\begin{proof}
It is stated in Remark 2.6 that $A_{1}^{\rho}\subseteq A_{1}^{\rho,loc}$. Hence, it suffices to show $\nu\in A_\infty^{\rho,loc}$. We first claim that $(1)$ $\mu\nu\in A_\infty^\rho$; $(2)$ $\mu^{-1}\in RH_\infty^\rho$; $(3)$ $RH_\infty^\rho\subseteq A_\infty^\rho$; $(4)$ $\tilde{\mu}\in A_\infty^\rho$, $\tilde{\omega}\in RH_{\infty}^\rho\Rightarrow\tilde{\mu}\tilde{\omega}\in A_\infty^\rho$.

Under the hypotheses of lemma, for every cube $Q$ and every measurable set $E$ contained in $Q$, there exist $\theta_1,\theta_2\geq0$ and $\epsilon,\delta>0$ such that
\begin{align*}
\frac{\mu\nu(E)}{\mu\nu(Q)}\lesssim\Big(1+\frac{r_Q}{\rho(x_Q)}\Big)^{\theta_1}
\Big(\frac{\mu(E)}{\mu(Q)}\Big)^\epsilon\lesssim\Big(1+\frac{r_Q}{\rho(x_Q)}\Big)^
{\theta_1+\epsilon\theta_2}\Big(\frac{|E|}{|Q|}\Big)^{\delta\epsilon},
\end{align*}
which implies that $\mu\nu\in A_\infty^{\rho}$.

Fix a cube $Q\subseteq\mathbb{R}^n$, by H\"{o}lder's inequality,
\begin{align*}
|Q|=\int_Q\mu(x)^{-1/2}\mu(x)^{1/2}\leq\Big(\int_Q\mu(x)^{-1}dx\Big)^{1/2}
\Big(\int_Q\mu(x)dx\Big)^{1/2}.
\end{align*}
Then there exists a $\theta\geq0$ such that
\begin{align*}
\esss_{x\in Q}\mu(x)^{-1}&\lesssim\frac{|Q|\psi_\theta(Q)}{\mu(Q)}\\
&\leq\frac{|Q|\psi_\theta(Q)}{|Q|^2}\Big(\int_Q\mu(x)^{-1}dx\Big)\\
&=\psi_\theta(Q)\langle\mu^{-1}\rangle_Q,
\end{align*}
this shows $\mu^{-1}\in RH_\infty^\rho$.

Let $E$ be any measurable subset of $Q$. For any $\omega\in RH_\infty^\rho$,
\begin{align*}
\frac{\omega(E)}{\omega(Q)}&=\frac{1}{\omega(Q)}\int_Q\chi_E\omega\\
&\leq\frac{|E|}{\omega(Q)}\esss_{x\in Q}\omega(x)\\
&\lesssim\psi_\theta(Q)\frac{|E|}{|Q|},
\end{align*}
so $\omega\in A_\infty^\rho$, which yields that $RH_\infty^\rho\subseteq A_\infty^\rho$.

In view of $(3)$, we have $\tilde{\mu}\in A_q^{\rho,\theta_3}$ and $\tilde{\omega}\in A_r^{\rho,\theta_4}$ for some $\theta_3, \theta_4\geq0$ and $q,r>1$. By $\tilde{\omega}\in RH_\infty^\rho$, there is a $\theta_5\geq0$ such that
\begin{align*}
\frac{1}{|Q|}\int_Q\tilde{\mu}(x)\tilde{\omega}(x)dx&\leq\frac{\esss_{x\in Q}\tilde{\omega}(x)}{|Q|}\int_Q\tilde{\mu}(x)dx\\
&\lesssim\psi_{\theta_5}(Q)\Big(\frac{1}{|Q|}\int_Q\tilde{\mu}(x)dx\Big)
\Big(\frac{1}{|Q|}\int_Q\tilde{\omega}(x)dx\Big).
\end{align*}
On the other hand, by H\"{o}lder's inequality,
\begin{align*}
&\Big(\frac{1}{|Q|}\int_Q(\tilde{\mu}(x)\tilde{\omega}(x))^{1-(q+r-1)'}dx\Big)^{(q+r-1)-1}\\
&\quad\leq\Big(\frac{1}{|Q|}\int_Q\tilde{\mu}(x)^{[1-(q+r-1)'][1+\frac{r-1}{q-1}]}dx\Big)
^{\frac{(q+r-1)-1}{1+(r-1)/(q-1)}}\\
&\qquad\times\Big(\frac{1}{|Q|}\int_Q\tilde{\omega}(x)^{[1-(q+r-1)'][1+\frac{r-1}{q-1}]'}dx\Big)
^{\frac{(q+r-1)-1}{(1+(r-1)/(q-1))'}}\\
&\quad=\Big(\frac{1}{|Q|}\int_Q\tilde{\mu}(x)^{1-q'}dx\Big)^{q-1}
\Big(\frac{1}{|Q|}\int_Q\tilde{\omega}(x)^{1-r'}dx\Big)^{r-1}.
\end{align*}
Thus,
\begin{align*}
&\Big(\frac{1}{|Q|}\int_Q\tilde{\mu}(x)\tilde{\omega}(x)dx\Big)
\Big(\frac{1}{|Q|}\int_Q(\tilde{\mu}(x)\tilde{\omega}(x))^{1-(q+r-1)'}dx\Big)^{(q+r-1)-1}\\
&\quad\lesssim\psi_{\theta_3+\theta_4+\theta_5}(Q)\Big(\frac{1}{|Q|}\int_Q\tilde{\mu}(x)dx\Big)
\Big(\frac{1}{|Q|}\int_Q\tilde{\mu}(x)^{1-q'}dx\Big)^{q-1}\psi_{\theta_3}(Q)^{-1}\\
&\qquad\times\Big(\frac{1}{|Q|}\int_Q\tilde{\omega}(x)dx\Big)
\Big(\frac{1}{|Q|}\int_Q\tilde{\omega}(x)^{1-r'}dx\Big)^{r-1}\psi_{\theta_4}(Q)^{-1}\\
&\quad\leq\psi_{\theta_3+\theta_4+\theta_5}(Q)[\tilde{\mu}]_{A_q^{\rho,\theta_3}}
[\tilde{\omega}]_{A_r^{\rho,\theta_4}},
\end{align*}
which implies that $\tilde{\mu}\tilde{\omega}\in A_{q+r-1}^{\rho,\theta_3+\theta_4+\theta_5}$. So, $\tilde{\mu}\tilde{\omega}\in A_\infty^\rho$.

Now we return to prove this lemma. Since $\mu\nu\in A_\infty^\rho$ by $(1)$, there exist $s,p>1$ such that $\mu\nu\in A_p^\rho\cap RH_s^\rho$, we have $\mu\nu=\omega_1\omega_2$, where $\omega_1\in A_1^\rho$ and $\omega_2\in RH_\infty^\rho$ by Lemma \ref{property of reverse holder class}. As a consequence of Lemma \ref{property of reverse holder class} and $(2)$, we have $\omega_2\mu^{-1}\in RH_\infty^\rho$. In virtue of $(4)$ and $\omega_1\in A_1^\rho$, $\nu=\omega_1\omega_2\mu^{-1}\in A_\infty^\rho\subseteq A_\infty^{\rho,loc}$. This completes the proof of Lemma \ref{strong condition implies weak condition}.
\end{proof}

Now, we are in the position to prove Theorem \ref{mixed weak type inequality for variation of heat semigroup in Sch}.
\begin{proof}[Proof of Theorem \ref{mixed weak type inequality for variation of heat semigroup in Sch}]
For any $f\in C_c^\infty(\mathbb{R}^n)$ and $x\in\mathbb{R}^n$, we write $f=f_1+f_2$, where $f_1:=f\chi_{B_x}$ with $B_x:=B(x,\rho(x))$. Then the sublinearity of $\mathcal{V}_\varrho(e^{-t\mathcal{L}})$ yields that
\begin{align}\label{f1+f2}
&\mu\nu\Big(\Big\{x\in\mathbb{R}^n:\frac{\mathcal{V}_\varrho(e^{-t\mathcal{L}})(f\nu)(x)}{\nu(x)}>
\lambda\Big\}\Big)\\
&\quad\leq\mu\nu\Big(\Big\{x\in\mathbb{R}^n:\frac{\mathcal{V}_\varrho(e^{-t\mathcal{L}})(f_1\nu)(x)}
{\nu(x)}>\frac{\lambda}{2}\Big\}\Big)\nonumber\\
&\qquad+
\mu\nu\Big(\Big\{x\in\mathbb{R}^n:\frac{\mathcal{V}_\varrho(e^{-t\mathcal{L}})
(f_2\nu)(x)}{\nu(x)}>\frac{\lambda}{2}\Big\}\Big)\nonumber.
\end{align}

We claim that for some $\sigma>0$ given as in Theorem \ref{mixed weak type of CZ and HL}, there holds
\begin{align}\label{f1 pointwise}
\mathcal{V}_\varrho(e^{-t\mathcal{L}})(f_1)(x)\leq\mathcal{V}_\varrho(e^{-t\Delta})(f_1)(x)+
M^{\rho,\sigma}f(x).
\end{align}
and
\begin{align}\label{f2 pointwise}
\mathcal{V}_\varrho(e^{-t\mathcal{L}})(f_2)(x)\lesssim M^{\rho,\sigma} f(x).
\end{align}
If we are done, by \eqref{f1 pointwise}-\eqref{f2 pointwise}, Theorem \ref{mixed weak type of CZ and HL} and Lemmas \ref{mixed weak type inequality for local}, \ref{strong condition implies weak condition}, we have
\begin{align*}
&\mu\nu\Big(\Big\{x\in\mathbb{R}^n:\frac{\mathcal{V}_\varrho(e^{-t\mathcal{L}})(f\nu)(x)}{\nu(x)}>
\lambda\Big\}\Big)\\
&\quad\lesssim\mu\nu\Big(\Big\{x\in\mathbb{R}^n:\frac{M^{\rho,\sigma}(f\nu)(x)}{\nu(x)}>
\frac{\lambda}{4}\Big\}\Big)
+\mu\nu\Big(\Big\{x\in\mathbb{R}^n:\frac{\mathcal{V}_\varrho(e^{-t\Delta})(f_1\nu)(x)}{\nu(x)}>
\frac{\lambda}{4}\Big\}\Big)\\
&\quad\lesssim\lambda^{-1}\int_{\mathbb{R}^n}|f(x)|\mu(x)\nu(x)dx.
\end{align*}

In what follows, we are going to prove the claim. We split $\mathcal{V}_\rho(e^{-t\mathcal{L}})(f_1)$ into two terms:
\begin{align}\label{f1 fenjie}
\mathcal{V}_\varrho(e^{-t\mathcal{L}})(f_1)(x)\leq\mathcal{V}_\varrho(e^{-t\Delta})(f_1)(x)+
\mathcal{V}_\varrho(e^{-t\mathcal{L}}-e^{-t\Delta})(f_1)(x).
\end{align}
Compared to \eqref{f1 pointwise}, it remains to estimate $\mathcal{V}_\varrho(e^{-t\mathcal{L}}-e^{-t\Delta})(f_1)$. By the perturbation formula in \cite{Dz}, there exists $\delta\in(0,\infty)$ such that for any $f\in C_c^\infty(\Rn)$ and $x\in\Rn$,
\begin{align*}
\Big|\Big(t\frac{\partial}{\partial t}e^{-t\mathcal{L}}-t\frac{\partial}{\partial t}e^{-t\Delta}\Big)(f_1)(x)\Big|\lesssim
\int_{B_x}t^{-n/2}\Big(\frac{\sqrt{t}}{\sqrt{t}+\rho(x)}\Big)
^\delta e^{-c|x-y|^2/t}|f(y)|dy.
\end{align*}
From this, we then have
\begin{align}\label{L1+L2}
&\mathcal{V}_\rho(e^{-t\mathcal{L}}-e^{-t\Delta})(f_1)(x)\\
&\quad=\sup_{\{t_i\}\downarrow0}
\Big(\sum_{i=1}^\infty|e^{-t_i\mathcal{L}}f_1(x)-e^{-t_{i+1}\mathcal{L}}f_1(x)-
e^{-t_i\triangle}f_1(x)+e^{-t_{i+1}\triangle}f_1(x)|^\varrho\Big)^{1/\varrho}\nonumber\\
&\quad\leq\sup_{\{t_i\}\downarrow0}
\sum_{i=1}^\infty|e^{-t_i\mathcal{L}}f_1(x)-e^{-t_{i+1}\mathcal{L}}f_1(x)-
e^{-t_i\triangle}f_1(x)+e^{-t_{i+1}\triangle}f_1(x)|\nonumber\\
&\quad=\sup_{\{t_i\}\downarrow0}
\sum_{i=1}^\infty\Big|\int_{t_{i+1}}^{t_i}t\frac{\partial}{\partial t}(e^{-t\mathcal{L}}-e^{-t\triangle})(f_1)(x)
\frac{dt}{t}\Big|\nonumber\\
&\quad\leq\int_{0}^{\infty}
\Big|\Big(t\frac{\partial}{\partial t}e^{-t\mathcal{L}}-t\frac{\partial}{\partial t}e^{-t\Delta}\Big)(f_1)(x)\Big|\frac{dt}{t}\nonumber\\
&\quad\lesssim\int_{0}^{\rho(x)^2}\int_{B_x}t^{-n/2}\Big(\frac{\sqrt{t}}{\sqrt{t}+\rho(x)}\Big)
^\delta
e^{-c|x-y|^2/t}|f(y)|dy\frac{dt}{t}\nonumber\\
&\qquad+\int_{\rho(x)^2}^{\infty}\int_{B_x}t^{-n/2}\Big(\frac{\sqrt{t}}{\sqrt{t}+\rho(x)}\Big)
^\delta
e^{-c|x-y|^2/t}|f(y)|dy\frac{dt}{t}\nonumber\nonumber\\
&\quad=:L_1(x)+L_2(x)\nonumber.
\end{align}
A direct computation shows that
\begin{align}\label{L2}
L_2(x)&\leq\int_{\rho(x)^2}^{\infty}\int_{B_x}t^{-n/2}|f(y)|dy\frac{dt}{t}\\
&\lesssim\int_{\rho(x)^2}^{\infty}\frac{2^\sigma}{(1+\rho(x)/\rho(x))^\sigma|B_x|}
\int_{B_x}(\rho(x)/\sqrt{t})^n|f(y)|dy\frac{dt}{t}\nonumber\\
&\leq M^{\rho,\sigma} f(x)\int_{\rho(x)^2}^{\infty}(\rho(x)/\sqrt{t})^n\frac{dt}{t}\sim M^{\rho,\sigma} f(x).\nonumber
\end{align}
To deal with the term $L_1(x)$, we first claim that for $t<\rho(x)^2$,
\begin{align*}
\int_{B_x}t^{-n/2}e^{-c|x-y|^2/t}|f(y)|dy\lesssim M^{\rho,\sigma} f(x).
\end{align*}
Assume we have done, then
\begin{align}\label{L1}
L_1(x)&\lesssim M^{\rho,\sigma} f(x)\int_{0}^{\rho(x)^2}\Big(\frac{\sqrt{t}}{\sqrt{t}+\rho(x)}\Big)^\delta\frac{dt}{t}\\
&\leq M^{\rho,\sigma} f(x)\int_{0}^{\rho(x)^2}\frac{t^{\delta/2-1}}{\rho(x)^\delta}dt\sim M^{\rho,\sigma} f(x)\nonumber.
\end{align}
Now we return to prove the claim. Using $t<\rho(x)^2$ and dyadic annuli-type decomposition, we have
\begin{align*}
&\int_{B_x}t^{-n/2}e^{-c|x-y|^2/t}|f(y)|dy\\
&\quad\leq\int_{\Rn}t^{-n/2}e^{-c|x-y|^2/t}|f(y)|dy\\
&\quad\leq\int_{B(x,\sqrt{t})}t^{-n/2}|f(y)|dy+\sum_{k=0}^{\infty}
\int_{B(x,2^{k+1}\sqrt{t})\backslash B(x,2^k\sqrt{t})}t^{-n/2}e^{-c2^{2k}}|f(y)|dy\\
&\quad=\sum_{k=0}^{\infty}
2^{(k+1)n}e^{-c2^{2k}}\frac{\Big(1+\frac{2^{k+1}\sqrt{t}}{\rho(x)}\Big)
^\sigma}{2^{(k+1)n}t^{n/2}
\Big(1+\frac{2^{k+1}\sqrt{t}}{\rho(x)}\Big)^\sigma}\int_{B(x,2^{k+1}\sqrt{t})}|f(y)|dy\\
&\qquad+\frac{\Big(1+\frac{\sqrt{t}}{\rho(x)}\Big)^\sigma}{t^{n/2}
\Big(1+\frac{\sqrt{t}}{\rho(x)}\Big)^\sigma}\int_{B(x,\sqrt{t})}|f(y)|dy\\
&\quad\lesssim M^{\rho,\sigma}f(x).
\end{align*}
This proves the claim. Then we prove \eqref{f1 pointwise} by combining with \eqref{f1 fenjie}-\eqref{L1}.

Finally, we prove \eqref{f2 pointwise}. A similar estimate of \eqref{L1+L2}, we have
\begin{align*}
\mathcal{V}_\varrho(e^{-t\mathcal{L}})(f_2)(x)\leq\int_{0}^{\infty}\Big|t\frac{\partial}{\partial t}e^{-t\mathcal{L}}f_2(x)\Big|\frac{dt}{t}.
\end{align*}
Applying
$$\Big|\frac{\partial}{\partial t}p_t(x,y)\Big|\lesssim\frac{1}{t^{n/2+1}}e^{-c|x-y|^2/t}\Big(1+\frac{\sqrt{t}}
{\rho(x)}+\frac{\sqrt{t}}{\rho(y)}\Big)^{-N}, N>0,$$
we deduce that
\begin{align*}
&\int_{0}^{\infty}\Big|t\frac{\partial}{\partial t}e^{-t\mathcal{L}}f_2(x)\Big|\frac{dt}{t}\\
&\quad=\int_{0}^{\infty}\Big|\int_{|x-y|>\rho(x)}t\frac{\partial}{\partial t}p_{t}(x,y)f(y)dy\Big|\frac{dt}{t}\\
&\quad\lesssim\int_{0}^{\infty}t^{-n/2}(1+\sqrt{t}/\rho(x))^{-N}\int_{|x-y|>\rho(x)}
e^{-c|x-y|^2/t}|f(y)|dy\frac{dt}{t}\\
&\quad\leq\int_{0}^{\infty}t^{-n/2}(1+\sqrt{t}/\rho(x))^{-N}\sum_{j=0}^{\infty}
e^{-c[2^j\rho(x)]^2/t}\int_{B(x,2^{j+1}\rho(x))\backslash B(x,2^{j}\rho(x))}|f(y)|dy\frac{dt}{t}\\
&\quad\lesssim\int_{0}^{\infty}t^{-n/2}(1+\sqrt{t}/\rho(x))^{-N}\sum_{j=0}^{\infty}
\Big(\frac{\sqrt{t}}{2^j\rho(x)}\Big)^A2^{j\sigma}(2^j\rho(x))^n\Big(1+\frac{2^{j+1}\rho(x)}
{\rho(x)}\Big)^{-\sigma}\\
&\qquad\times\frac{1}{|B(x,2^{j+1}\rho(x))|}\int_{B(x,2^{j+1}\rho(x))}|f(y)|dy\frac{dt}{t}\\
&\quad\lesssim M^{\rho,\sigma} f(x)\int_{0}^{\infty}(1+\sqrt{t}/\rho(x))^{-N}(\sqrt{t}/\rho(x))^{A-n}\sum_{j=0}^{\infty}
2^{-j(A-n-\sigma)}\frac{dt}{t}\\
&\quad\lesssim M^{\rho,\sigma} f(x)\int_{0}^{\infty}(1+\sqrt{t}/\rho(x))^{-N}(\sqrt{t}/\rho(x))^{A-n}\frac{dt}{t}\\
&\quad=M^{\rho,\sigma} f(x)\int_{0}^{\infty}\frac{s^{A-n-1}}{(1+s)^N}ds\lesssim M^{\rho,\sigma} f(x),
\end{align*}
where $A,N\in(0,\infty)$ such that $A>n+\sigma$ and $N>A-n$. This proves \eqref{f2 pointwise}.
\end{proof}

\section{Characterizations of new classes of weights via restricted weak type estimates of maximal operators}
In this section, we are going to show the characterizations of $A_{p,q,\alpha}^{\rho,\theta,\mathcal{R}}$ via restricted weak type estimates of maximal operators. We will employ the ideas of Roure Perdices in \cite{R}, but the classes of weights we deal with are larger than ones in \cite{R}. Before we prove the main theorems, we first show that \eqref{restricted weak type for maximal function} in fact implies $1/p-1/q\leq\alpha/n$.
\begin{theorem}\label{th5.1}
Let $0<\alpha<n$ and $\mu$, $\nu$ be weights. If $M_{\alpha}^{\rho,\theta}$ is bounded from $L^{p,1}(\mu)$ to $L^{q,\infty}(\nu)$, then $1/p-1/q\leq\alpha/n$.
\end{theorem}
\begin{proof}
We use a contradiction argument by assuming $1/p-1/q>\alpha/n$. Fix a cube $Q$, let $f=\chi_Q$. Then for $x\in Q$,
\begin{align*}
M_{\alpha}^{\rho,\theta}f(x)\geq\frac{1}{(|Q|\psi_\theta(Q))^{1-\frac{\alpha}{n}}}|Q|=
\frac{|Q|^{\frac{\alpha}{n}}}{\psi_\theta(Q)^{1-\frac{\alpha}{n}}}.
\end{align*}
Then
\begin{align*}
\Big\|\frac{|Q|^{\frac{\alpha}{n}}}{\psi_\theta(Q)^{1-\frac{\alpha}{n}}}\chi_Q\Big\|_{L^{q,\infty}(\nu)}
\leq\|M_{\alpha}^{\rho,\theta}f\|_{L^{q,\infty}(\nu)}\lesssim\|\chi_Q\|_{L^{p,1}(\mu)}.
\end{align*}
Note that
\begin{equation*}
\|\chi_Q\|_{L^{p,1}(\mu)}=p\mu(Q)^{1/p},
\end{equation*}
we then have
\begin{equation*}
\Big(\frac{1}{|Q|}\int_Q\nu(x)dx\Big)^{1/q}\lesssim|Q|^{1/p-1/q-\alpha/n}\psi_\theta(Q)^{1-\alpha/n}
\Big(\frac{1}{|Q|}\int_Q\mu(x)dx\Big)^{1/p}.
\end{equation*}

Now let $x_0$ be any Lebesgue point of $\mu$ and $\nu$. Let $Q_j$ be a sequence of cubes centered at $x_0$ that shrink to $x_0$. Since $1/p-1/q>\alpha/n$ and $0<\alpha<n$, by the Lebesgue differential theorem,
\begin{equation*}
|Q|^{1/p-1/q-\alpha/n}\psi_\theta(Q)^{1-\alpha/n}
\Big(\frac{1}{|Q|}\int_Q\mu(x)dx\Big)^{1/p}\rightarrow0,
\end{equation*}
which shows that $\nu=0$. This deduces a contradiction and completes the proof of Theorem \ref{th5.1}.
\end{proof}

\begin{proof}[Proof of Theorem \ref{characterization restricted weak type}]
We first prove the necessity of Theorem \ref{characterization restricted weak type}. Fix a cube $Q$ and $\epsilon>1$. By \cite[Theorem 1.4.16.(v)]{Gra}, there is a non-negative function $f$ such that $\|f\|_{L^{p,1}(\mu)}=1$ and
\begin{equation*}
\frac{1}{\epsilon p}\|\chi_Q\mu^{-1}\|_{L^{p',\infty}(\mu)}\leq\int_Qf(\chi_Q\mu^{-1})\mu=\int_Qf.
\end{equation*}
According to the definition of $M_{\alpha}^{\rho,\theta}$,
\begin{equation*}
\Big(\frac{1}{|Q|^{1-\frac{\alpha}{n}}}\Big(1+\frac{r_Q}{\rho(x_Q)}\Big)^{-\theta(1-\frac{\alpha}{n})}
\int_Q|f|\Big)\chi_Q\leq M_{\alpha}^{\rho,\theta}f.
\end{equation*}
Then
\begin{equation*}
\Big(\frac{1}{\epsilon p|Q|^{1-\frac{\alpha}{n}}}\Big(1+\frac{r_Q}{\rho(x_Q)}\Big)^{-\theta(1-\frac{\alpha}{n})}
\|\chi_Q\mu^{-1}\|_{L^{p',\infty}(\mu)}\Big)\chi_Q\leq M_{\alpha}^{\rho,\theta}f.
\end{equation*}
From this and \eqref{restricted weak type for maximal function}, it follows that
\begin{equation*}
\nu(Q)^{1/q}\frac{\|\chi_Q\mu^{-1}\|_{L^{p',\infty}(\mu)}}{
(|Q|\psi_\theta(Q))^{1-\frac{\alpha}{n}}}\leq c\epsilon p.
\end{equation*}
Therefore
\begin{equation*}
[\mu,\nu]_{A_{p,q,\alpha}^{\rho,\theta,\mathcal{R}}}\leq pc<\infty.
\end{equation*}

Next, we turn to prove the sufficiency. Let $\lambda>0$ and $K>0$. Denote
\begin{equation*}
E_\lambda:=\{x\in\mathbb{R}^n:M_{\alpha}^{\rho,\theta}f(x)>\lambda\},\quad E_{\lambda,K}:=B(0,K)\cap E_\lambda.
\end{equation*}
By the definition of $M_{\alpha}^{\rho,\theta}$, for any $x\in E_{\lambda,K}$, there is a sequence $\{x_j\}_j$ such that $E_{\lambda,K}\subset\bigcup_j Q_{x_j}$, where $\{Q_{x_j}\}$ is bounded overlapping. Then by H\"{o}lder's inequality,
\begin{align*}
\nu(E_{\lambda,K})&\leq\sum_j\nu(Q_{x_j})\\
&\leq\lambda^{-q}\sum_j\nu(Q_{x_j})
\Big(\Big[\Big(1+\frac{r_{Q_{x_j}}}{\rho(Q_{x_j})}\Big)^\theta|Q_{x_j}|\Big]
^{\frac{\alpha}{n}-1}\int_{Q_{x_j}}|f(y)|dy\Big)^q\\
&\leq\lambda^{-q}\sum_j\nu(Q_{x_j})
\Big[\Big(1+\frac{r_{Q_{x_j}}}{\rho(Q_{x_j})}\Big)^\theta|Q_{x_j}|\Big]
^{q(\frac{\alpha}{n}-1)}\|\chi_{Q_{x_j}}\mu^{-1}\|_{L^{p',\infty}}^q
\|f\chi_{Q_{x_j}}\|_{L^{p,1}(\mu)}^q\\
&\leq\lambda^{-q}[\mu,\nu]_{A_{p,q,\alpha}^{\rho,\theta,\mathcal{R}}}^q\sum_j
\|f\chi_{Q_{x_j}}\|_{L^{p,1}(\mu)}^q\\
&\lesssim\lambda^{-q}[\mu,\nu]_{A_{p,q,\alpha}^{\rho,\theta,\mathcal{R}}}^q\|f\|_{L^{p,1}(\mu)}^q,
\end{align*}
where we use Lemma \ref{summation lemma} in the last inequality. Thus, let $K\rightarrow\infty$, the monotone convergence theorem leads to the desired result.
\end{proof}

\begin{proof}[Proof of Theorem \ref{another characterization restricted weak type}]
Our proof of Theorem \ref{another characterization restricted weak type} goes the following:
\begin{equation*}
(1)\Rightarrow(2)\Rightarrow(3)\Rightarrow(4)\Rightarrow(1).
\end{equation*}
It is obvious that $(1)\Rightarrow(2)$. Next, we show that $(2)$ implies $(3)$. Fix a cube $Q$ and a measurable set $P$ with $0<\mu(P)<\infty$. Observe that
\begin{equation*}
\frac{|P\cap Q|}{(\psi_\theta(Q)|Q|)^{1-\frac{\alpha}{n}}}\chi_Q\leq M_\alpha^{\rho,\theta}(\chi_P).
\end{equation*}
Then the condition (2) yields that
\begin{equation*}
\nu(Q)^{1/q}\frac{|P\cap Q|}{(\psi_\theta(Q)|Q|)^{1-\frac{\alpha}{n}}}\lesssim\mu(P)^{1/p},
\end{equation*}
which implies that
\begin{equation*}
\|\mu,v\|_{A_{p,q,\alpha}^{\rho,\theta,\mathcal{R}}}<\infty.
\end{equation*}

Since $(4)\Rightarrow(1)$ is a consequence of Theorem \ref{characterization restricted weak type}, it remains to prove $(3)\Rightarrow(4)$. Fix a cube $Q$. It is not hard to check that
\begin{equation*}
\sup_{0<\mu(P)<\infty}\frac{|P\cap Q|}{\mu(P)^{1/p}}=\sup_{P\subset Q}\frac{|P |}{\mu(P)^{1/p}}.
\end{equation*}
Hence, for any $P\subset Q$,
\begin{equation*}
|P|\leq\|\mu,\nu\|_{A_{p,q,\alpha}^{\rho,\theta,\mathcal{R}}}
\frac{\mu(P)^{1/p}(\psi_\theta(Q)|Q|)^{1-\frac{\alpha}{n}}}{v(Q)^{1/q}}.
\end{equation*}
Let $P=\{x\in Q:\mu(x)^{-1}>y\}$. Then
\begin{align}\label{still valid for p=1}
y\mu(P)&=\int_Py\mu(x)dx\leq\int_P\mu(x)\mu(x)^{-1}dx\\
&=|P|\leq\|\mu,\nu\|_{A_{p,q,\alpha}^{\rho,\theta,\mathcal{R}}}
\frac{\mu(P)^{1/p}(\psi_\theta(Q)|Q|)^{1-\frac{\alpha}{n}}}{\nu(Q)^{1/q}}.
\nonumber
\end{align}

For $p>1$, we get that
\begin{align*}
\nu(Q)^{1/q}\frac{\|\chi_Q\mu^{-1}\|_{L^{p',\infty}(\mu)}}
{(\psi_\theta(Q)|Q|)^{1-\frac{\alpha}{n}}}
\leq\|\mu,\nu\|_{A_{p,q,\alpha}^{\rho,\theta,\mathcal{R}}},
\end{align*}
which yields that $(\mu,\nu)\in A_{p,q,\alpha}^{\rho,\theta,\mathcal{R}}$.

If $p=1$. For $y<\|\chi_Q\mu^{-1}\|_{L^{\infty}(\mu)}$, using \eqref{still valid for p=1},
\begin{equation*}
y\leq\|\mu,\nu\|_{A_{1,q,\alpha}^{\rho,\theta,\mathcal{R}}}
\frac{(\psi_\theta(Q)|Q|)^{1-\frac{\alpha}{n}}}{\nu(Q)^{1/q}}.
\end{equation*}
Letting $y$ increase to $\|\chi_Q\mu^{-1}\|_{L^{\infty}(\mu)}$, we obtain
\begin{align*}
\nu(Q)^{1/q}\frac{\|\chi_Q\mu^{-1}\|_{L^{\infty}(\mu)}}
{(\psi_\theta(Q)|Q|)^{1-\frac{\alpha}{n}}}
\leq\|\mu,\nu\|_{A_{1,q,\alpha}^{\rho,\theta,\mathcal{R}}},
\end{align*}
which also shows that $(\mu,\nu)\in A_{1,q,\alpha}^{\rho,\theta,\mathcal{R}}$.
\end{proof}

\section{Quantitative restricted weak type estimate for variation operator}

This section is devoted to proving Theorem \ref{quantitative restricted weak type}. Note that when $\mathcal{L}$ is a Schr\"{o}dinger operator, the kernel of $e^{-t\mathcal{L}}$ satisfies \eqref{upper bound}, and it follows from \cite{BetFHR} that $\mathcal{V}_\varrho(e^{-t\mathcal{L}})$ is of weak type $(1,1)$ and bounded on $L^2(\mathbb{R}^n)$. Therefore, Theorem \ref{quantitative restricted weak type} can be regarded as a special case of the following general theorem.

\begin{theorem}\label{th6.1}
Let $\varrho>2$, $\theta\geq0$ and $\rho$ be a critical radius function. Let $\mathcal{L}$ be defined as in Section 3. Suppose that $\mathcal{V}_\varrho(e^{-t\mathcal{L}})$ is of weak type $(1,1)$ and bounded on $L^2(\mathbb{R}^n)$. Then for $1<p<\infty$ and $\omega\in A_{p}^{\rho,\theta,\mathcal{R}}$, we have
$$\|\mathcal{V}_\varrho(e^{-t\mathcal{L}})f\|_{L^{p,\infty}(\omega)}
\lesssim[\omega]_{A_{p}^{\rho,\theta,\mathcal{R}}}^{p+1}\|f\|_{L^{p,1}(\omega)}.$$
\end{theorem}

In order to prove Theorem \ref{th6.1}, we first establish the following quantitative restricted weak type estimate for sparse operator associated with critical radius function.
\begin{theorem}\label{restricted weak type estimate of sparse operator}
Let $1<p<\infty$ and $\omega\in A_{p}^{\rho,\theta,\mathcal{R}}$ with $\theta\geq0$. Then
\begin{equation*}
\|\mathcal{A}_{\mathcal{S}}^{\rho,N}f\|_{L^{p,\infty}(\omega)}\lesssim
[\omega]_{A_{p}^\mathcal{\rho,\theta,\mathcal{R}}}^{p+1}\|f\|_{L^{p,1}(\omega)},
\end{equation*}
provided that $N\geq(p+1)\theta$.
\end{theorem}
\begin{proof}
Since $\mathcal{S}$ is a sparse family, for any $Q\in\mathcal{S}$, there exists $E_Q\subset Q$ such that $|E_Q|\sim|Q|$ and $\{E_Q\}_{Q\in\mathcal{S}}$ is a disjoint family. Then
\begin{equation}\label{new classes weights appear}
\Big(1+\frac{r_Q}{\rho(x_Q)}\Big)^{-p\theta}\frac{\omega(3Q)}{\omega(E_Q)}
\sim\frac{|E_Q|^p}{|3Q|^p
\Big(1+\frac{r_Q}{\rho(x_Q)}\Big)^{p\theta}}
\frac{\omega(3Q)}{\omega(E_Q)}\leq\|\omega\|_{A_{p}^{\rho,\theta,\mathcal{R}}}^p.
\end{equation}
Observe that for any $x\in Q$,
\begin{equation*}
Q\subset Q(x,2r_Q)\subset3Q,
\end{equation*}
which implies that
\begin{equation}\label{dominated by center HL}
\Big(\frac{1}{\omega(3Q)}\int_Qg\omega\Big)\chi_Q(x)\leq
\Big(\frac{1}{\omega(Q(x,2r_Q))}\int_{Q(x,2r_Q)}g\omega\Big)\chi_Q(x)\leq
\inf_Q M_\omega^c(g)(x),
\end{equation}
where
\begin{align*}
M_\omega^cg(x):=\sup_{Q_x}\frac{1}{|Q_x|}\int_{Q_x}|g(x)|\omega(x)dx,
\end{align*}
and $Q_x$ denotes any cube centered at $x$.

By duality,
\begin{equation*}
\|\mathcal{A}_{\mathcal{S}}^{\rho,N}f\|_{L^{p,\infty}(\omega)}=
\underset{\|g\|_{L^{p',1}(\omega)=1}}{\sup_{0\leq g\in L^{p',1}(\omega)}}\int_{\mathbb{R}^n}
\mathcal{A}_{\mathcal{S}}^{\rho,N}f(x)g(x)\omega(x)dx.
\end{equation*}
Since
\begin{align*}
&\int_{\mathbb{R}^n}\mathcal{A}_{\mathcal{S}}^{\rho,N}(f)(x)g(x)\omega(x)dx\\
&\quad=\sum_{Q\in\mathcal{S}}\langle|f|\rangle_{3Q}\Big(1+\frac{r_Q}{\rho(x_Q)}\Big)^{-N}
\int_Qg(x)\omega(x)dx\\
&\quad=\sum_{Q\in\mathcal{S}}\langle|f|\rangle_{3Q}\Big(1+\frac{r_Q}{\rho(x_Q)}\Big)^{-N}
\Big(\frac{1}{\omega(3Q)}\int_Qg(x)\omega(x)dx\Big)\omega(E_Q)\frac{\omega(3Q)}{\omega(E_Q)}.
\end{align*}
By \eqref{new classes weights appear}, \eqref{dominated by center HL}, Theorem \ref{characterization restricted weak type}, H\"{o}lder's inequality in Lorentz space and taking $N\geq(p+1)\theta$, we arrive at
\begin{align*}
&\int_{\mathbb{R}^n}\mathcal{A}_{\mathcal{S}}^{\rho,N}(f)(x)g(x)\omega(x)dx\\
&\quad\lesssim\|\omega\|_{A_{p}^{\rho,\theta,\mathcal{R}}}^p\sum_{Q\in\mathcal{S}}
\langle|f|\rangle_{3Q}\Big(1+\frac{r_Q}{\rho(x_Q)}\Big)^{p\theta-N}
\Big(\frac{1}{\omega(3Q)}\int_Qg(x)\omega(x)dx\Big)\omega(E_Q)\\
&\quad\lesssim\|\omega\|_{A_{p}^{\rho,\theta,\mathcal{R}}}^p\sum_{Q\in\mathcal{S}}
\langle|f|\rangle_{3Q}\Big(1+\frac{3r_Q}{\rho(x_Q)}\Big)^{-\theta}
\Big(\frac{1}{\omega(3Q)}\int_{Q}g(x)\omega(x)dx\Big)\omega(E_Q)\\
&\quad\leq\|\omega\|_{A_{p}^{\rho,\theta,\mathcal{R}}}^p\sum_{Q\in\mathcal{S}}\int_{E_Q}
\langle|f|\rangle_{3Q}\Big(1+\frac{3r_Q}{\rho(x_Q)}\Big)^{-\theta}
\Big(\frac{1}{\omega(3Q)}\int_Qg(x)\omega(x)dx\Big)\omega(x)dx\\
&\quad\leq\|\omega\|_{A_{p}^{\rho,\theta,\mathcal{R}}}^p
\int_{\mathbb{R}^n}M^{\rho,\theta}f(x)M_\omega^cg(x)\omega(x)dx\\
&\quad\leq\|\omega\|_{A_{p}^{\rho,\theta,\mathcal{R}}}^p\|M^{\rho,\theta}f\|_{L^{p,\infty}(\omega)}
\|M_\omega^cg\|_{L^{p',1}(\omega)}\\
&\quad\lesssim[\omega]_{A_{p}^{\rho,\theta,\mathcal{R}}}^{p+1}\|f\|_{L^{p,1}(\omega)}.
\end{align*}
\end{proof}

Now, we prove Theorem \ref{th6.1}.
\begin{proof}[Proof of Theorem \ref{th6.1}]
Let $B_j$, $Q_j$ be given as in the proof of Theorem \ref{quantitative Ap}. For any $\lambda>0$, by \eqref{global estimate} and \eqref{local estimate}, we get that
\begin{align}\label{M1+M2}
&\lambda\omega\Big(\Big\{x\in\mathbb{R}^n:\mathcal{V}_\varrho(e^{-t\mathcal{L}})f(x)>\lambda\Big\}
\Big)^{1/p}\\
&\quad\leq\lambda\omega \Big(\Big\{x\in\bigcup_j B_j:\mathcal{V}_\varrho(e^{-t\mathcal{L}})
(f\chi_{3Q_j})(x)>\lambda/2\Big\}\Big)^{1/p}\nonumber\\
&\qquad+\lambda\omega\Big(\Big\{x\in\bigcup_j B_j:\mathcal{V}_\varrho(e^{-t\mathcal{L}})
(f\chi_{(3Q_j)^c})(x)>\lambda/2\Big\}\Big)^{1/p}\nonumber\\
&\quad\lesssim\lambda\Big(\sum_j\omega(\{x\in B_j:\mathcal{A}_{\mathcal{S}_j}^{\rho,N}
(f\chi_{3Q_j})(x)>\lambda/2\})\Big)^{1/p}\nonumber\\
&\qquad+\lambda\Big(\sum_j\omega(\{x\in B_j:M^{\rho,\theta}
f(x)>\lambda/2\})\Big)^{1/p}\nonumber\\
&\quad=:M_1+M_2.\nonumber
\end{align}

First, we estimate $M_1$. By Theorem \ref{restricted weak type estimate of sparse operator} with $N\geq(p+1)\theta$ and Lemma \ref{summation lemma}, we derive that
\begin{align}\label{M1}
M_1&\leq\Big(\sum_j\lambda^p\omega(\{x\in\mathbb{R}^n:
\mathcal{A}_{\mathcal{S}_j}^{\rho,N}(f\chi_{3Q_j})(x)>\lambda/2\}\Big)^{1/p}\\
&\lesssim\Big(\sum_j[\omega]_{A_{p}^{\rho,\theta,\mathcal{R}}}^{(p+1)p}\|f\chi_{3 Q_j}\|_{L^{p,1}(\omega)}^p\Big)^{1/p}\nonumber\\
&\lesssim[\omega]_{A_{p}^{\rho,\theta,\mathcal{R}}}^{(p+1)}\|f\|_{L^{p,1}(\omega)}.\nonumber
\end{align}

Next, we take care of $M_2$. Utilizing Theorem \ref{characterization restricted weak type} and Lemma \ref{covering lemma}, we obtain
\begin{align*}
M_2&=\lambda\Big(\sum_j\int_{\{x\in\mathbb{R}^n:M^{\rho,\theta}f(x)>\lambda/2\}}\chi_{B_j}(x)\omega(x)dx
\Big)^{1/p}\\
&\lesssim\lambda\Big(\int_{\{x\in\mathbb{R}^n:M^{\rho,\theta}f(x)>\lambda/2\}}\omega(x)dx\Big)^{1/p}
\\
&\lesssim[\omega]_{A_{p}^{\rho,\theta,\mathcal{R}}}\|f\|_{L^{p,1}(\omega)}.
\end{align*}
This, together with \eqref{M1+M2} and \eqref{M1}, implies the desired conclusion and completes the proof of Theorem \ref{th6.1}.
\end{proof}

\subsection*{Declaration of competing interest}
The authors declare that there is no competing interest.


\begin{thebibliography}{99}


\bibitem{AJS}M. Akcoglu, R.L. Jones and P. Schwartz,
Variation in probability, ergodic theory and analysis,
Illinois J. Math. 42 (1998), 154--177.

\bibitem{Au}P. Auscher,
On necessary and sufficient conditions for $L^p$-estimates of Riesz transforms associated to elliptic operators on $\mathbb{R}^n$ and related estimates,
Mem. Amer. Math. Soc. 186(871) (2007), 75pp.

\bibitem{BePrQu}F. Berra, G. Pradolini and P. Quijano,
Mixed inequalities for operators associated to critical radius functions with applications to Schr\"{o}dinger type operators,
Potential Anal. Doi: 10.1007/s11118-022-10049-2.

\bibitem{BetFHR}J.J. Betancor, J.C. Fari\~{n}a, E. Harbour and L. Rodr\'{i}guez-Mesa,
$L^p$-boundedness properities of variation operators in the Schr\"odinger setting,
Rev. Mat. Complut. 26(2) (2013), 485--534.

\bibitem{BeHuWuYa}J.J. Betancor, W. Hu, H. Wu and D. Yang,
Boundedness of oscillation and variation of semigroups associated with Bessel Schr\"{o}dinger operators,
Nonlinear Anal. 202 (2021), 32pp.

\bibitem{BlKu}S. Blunck and P.C. Kunstmann,
Calder\'{o}n-Zygmund theory for non-integral operators and the $H^\infty$ functional calculus,
Rev. Mat. Iberoamericana 19(3) (2003), 919--942.

\bibitem{BonHQ}B. Bongioanni, E. Harboure and P. Quijano,
Two weighted inequalities for operators associated to a critical radius function,
Illinois J. Math. 64(2) (2020), 227--259.

\bibitem{BonHS}B. Bongioanni, E. Harboure and O. Salinas,
Classes of weights related to Schr\"{o}dinger operators,
J. Math. Anal. Appl. 373(2) (2011), 563--579.

\bibitem{Bo} J. Bourgain,
Pointwise ergodic theorems for arithmetric sets,
Publ. Math. Inst. Hautes \'{E}tudes Sci. 69(1) (1989), 5--45.

\bibitem{Bui}T.A. Bui,
Boundedness of variation operators and oscillation operators for certain semigroups,
Nonlinear Anal. 106 (2014), 124--137.

\bibitem{BuiBD1}T.A. Bui, T.Q. Bui and X.T. Duong,
Quantitative weighted estimates for some singular integrals related to critical functions,
J. Geom. Anal. 31(10) (2021), 10215--10245.

\bibitem{BuiBD}T.A. Bui, T.Q. Bui and X.T. Duong,
Quantitative estimates for square functions with new class of weights,
Potential Anal. 57(4) (2022), 545--569.


\bibitem{CalRi}M. Caldarelli and I.P. Rivera-R\'{i}os,
A sparse approach to mixed weak type inequalities,
Math. Z. 296(1-2) (2020), 787--812.

\bibitem{CaJRW1}J.T. Campbell, R.L. Jones, K. Reinhold and M. Wierdl,
Oscillations and variation for the Hilbert transform,
Duke Math. J. 105(1) (2000), 59--83.

\bibitem{CaJRW2}J.T. Campbell, R.L. Jones, K. Reinhold and M. Wierdl,
Oscillations and variation for singular integrals in higher dimensions,
Trans. Amer. Math. Soc. 355(5) (2003), 2115--2137.

\bibitem{CaIRXY}M. Cao, G. Iba\~{n}ez-Firnkorn, I.P. Rivera-R\'{i}os, Q. Xue and K. Yabuta,
A class of multilinear bounded oscillation operators on measure spaces and applications,
Math. Ann. (2023), Doi:10.1007/s00208-023-02619-5.

\bibitem{CaSiZh}M. Cao, Z. Si and J. Zhang,
Weak and strong type estimates for square functions associated with operators,
J. Math. Anal. Appl. 527(1) (2023), 29pp.

\bibitem{CaXYa}M. Cao, Q. Xue and K. Yabuta,
Weak and strong type estimates for the multilinear pseudo-differential operators,
J. Funct. Anal. 278(10) (2020), 46pp.

\bibitem{CheDHL}Y. Chen, Y. Ding, G. Hong and H. Liu,
Weighted jump and variational inequalities for rough operators.
J. Funct. Anal. 275(8) (2018), 2446--2475.

\bibitem{CheDHL2}Y. Chen, Y. Ding, G. Hong and H. Liu,
Variational inequalities for the commutators of rough operators with BMO functions,
Sci. China Math. 64(11) (2021), 2437--2460.

\bibitem{CheHo}Y. Chen and G. Hong,
The $L^2$-boundedness of the variational Calder\'{o}n-Zygmund operators,
J. Geom. Anal. 33(2) (2023), 28pp.

\bibitem{CheHoLi}Y. Chen, G. Hong and J. Li,
Quantitative weighted bounds for the  q-variation of singular integrals with rough kernels,
J. Fourier Anal. Appl. 29(3) (2023), 50pp.

\bibitem{CrMacMTV}R. Crescimbeni, R.A. Mac\'{i}as, T. Men\'{a}rguez, J.L. Torrea and B. Viviani,
The $\rho$-variation as an operator between maximal operators and singular integrals,
J. Evol. Equ. 9(1) (2009), 81--102.

\bibitem{CrMP}D. Cruz-Uribe, J.M. Martell and C. P\'{e}rez,
Weighted weak-type inequalities and a conjecture of Sawyer,
Int. Math. Res. Not. IMRN 30 (2005), 1849--1871.

\bibitem{CMP11}D. Cruz-Uribe, J.M. Martell and C. P\'{e}rez,
Weights, extrapolation and the theory of Rubio de Francia. Operator Theory: Advances and Applications,
Birkh\"{a}user/Springer Basel AG, Basel, 215 (2011), xiv+280 pp. ISBN: 978--3--0348--0071--6.

\bibitem{CruzM}D. Cruz-Uribe and K. Moen,
A fractional Muckenhoupt-Wheeden theorem and its consequences,
Integral Equations Operator Theory 76(3) (2013), 421--446.

\bibitem{CruzPe}D. Cruz-Uribe and C. P\'{e}rez,
On the two-weight problem for singular integral operators,
Ann. Sc. Norm. Super. Pisa Cl. Sci. (5) 1(4) (2002), 821--849.


\bibitem{DHL}Y. Ding, G. Hong and H. Liu,
Jump and variational inequalities for rough operators,
J. Fourier. Anal. Appl. 23(3) (2017), 679--711.

\bibitem{DuLY}X.T. Duong, J. Li and D. Yang,
Variation of Calder\'{o}n-Zygmund operators with matrix weight,
Commun. Contemp. Math. doi: 10.1142/S0219199720500625.

\bibitem{DzZ}J. Dziuba\'{n}ski and J. Zinkiewicz,
Hardy spaces $H^1$ associated to Schr\"{o}dinger operators with potntial satisfying reverse H\"{o}lder inequality,
Rev. Mat. Iberoam. 15(2) (1999), 279--296.

\bibitem{Dz}J. Dziuba\'{n}ski,
Note on $H^1$ spaces related to degenerate Schr\"{o}dinger operators,
Illinois J. Math. 49 (2005), 1271--1297.

\bibitem{GT}T.A. Gillespie and J.L. Torrea,
Dimension free estimates for the oscillation of Riesz transforms,
Israel J. Math. 141 (2004), 125-144.

\bibitem{Gra}L. Grafakos,
Classical Fourier analysis,
Third edition. Graduate Texts in Mathematics,
Springer, New York, (2014), pp. xviii+638. ISBN: 978-1-4939-1194-3.

\bibitem{GuWWY}W. Guo, Y. Wen, H. Wu and D. Yang,
Variational characterizations of weighted Hardy spaces and weighted $BMO$ spaces,
Proc. Roy. Soc. Edinburgh Sect. A 152(6) (2022), 1613--1632.

\bibitem{HyKa}T.P. Hyt\"{o}nen and A. Kairema,
Systems of dyadic cubes in a doubling metric space,
Colloq. Math. 126(1) (2012), 1--33.

\bibitem{HyLP}T.P. Hyt\"{o}nen, M.T. Lacey and C. P\'{e}rez,
Sharp weighted bounds for the q-variation of singular integrals,
Bull. London Math. Soc. 45(3) (2013), 529--540.

\bibitem{HyPe}T.P. Hyt\"{o}nen and C. P\'{e}rez,
Sharp weighted bounds involving $A_\infty$,
Anal. PDE 6(4) (2013), 777--818.


\bibitem{Jo}R.L. Jones,
Ergodic theory and connections with analysis and probability,
New York J. Math. 3A (1997/98), Proceedings of the New York Journal of Mathematics Conference, 31--67.

\bibitem{JR}R.L. Jones and K. Reinhold,
Oscillation and variation inequalities for convolution powers,
Ergodic Theory Dynam. Sys. 21(6) (2001), 1809--1829.


\bibitem{Ka}S. Kakaroumpas,
Two-weight estimates for sparse square functions and the separated bump conjecture,
Trans. Amer. Math. Soc. 375(5) (2022), 3003--3037.

\bibitem{KT}R.A. Kerman and A. Torchinsky,
Integral inequalities with weights for the Hardy maximal function,
Studia Math. 71(3) (1981/82), 277--284.

\bibitem{Le}D. L\'{e}pingle,
La variation d'order $p$ des semi-martingales,
Z. Wahrscheinlichkeitstheorie und Verw. Gebiete, 36(4) (1976), 295--316.

\bibitem{Ler1}A.K. Lerner,
On an estimate of Calder\'{o}n-Zygmund operators by dyadic positive operators,
J. Anal. Math. 121 (2013), 141--161.

\bibitem{Ler3}A.K. Lerner,
On pointwise estimates involving sparse operators,
New York J. Math. 22 (2017), 341--349.

\bibitem{LiOB}K. Li, S. Ombrosi and P.M. B\'{e}len,
Weighted mixed weak-type inequalities for multilinear operators,
Studia Math. 244(2) (2019), 203--215.

\bibitem{LRW}J. Li, R. Rahm, and B.D. Wick,
$A_p$ weights and quantitative estimates in the Schr\"{o}dinger setting,
Math. Z. 293(1-2) (2019), 259--283.

\bibitem{LiOP}K. Li, S. Ombrosi and C. P\'{e}rez,
Proof of an extension of E. Sawyer's conjecture about weighted mixed weak-type estimates,
Math. Ann. 374(1-2) (2019), 907--929.

\bibitem{Liu}H. Liu,
Variational characterization of $H^p$,
Proc. Roy. Soc. Edinburgh Sect. A 149(5) (2019), 1123--1134.

\bibitem{LW}F. Liu and H. Wu,
A criterion on oscillation and variation for the commutators of singular integrals,
Forum Math. 27 (2015), 77--97.

\bibitem{MaTX1}T. Ma, J.L. Torrea and Q. Xu,
Weighted variation inequalities for differential operators and singular integrals,
J. Funct. Anal. 268(2) (2015), 376--416.

\bibitem{MaTX2}T. Ma, J.L. Torrea and Q. Xu,
Weighted variation inequalities for differential operators and singular integrals in higher dimensions,
Sci. China Math. 268(2) (2015), 376--416.

\bibitem{MuWh}B. Muckenhoupt and R.L. Wheeden,
Some weighted weak-type inequalities for the Hardy-Littlewood maximal function and the Hilbert transform,
Indiana Univ. Math. J. 26(5) (1977), 801--816.

\bibitem{Perez}C. P\'{e}rez,
On sufficient conditions for the boundedness of the Hardy-Littlewood maximal operator between weighted $L^p$-spaces with different weights,
Proc. London Math. Soc. (3) 71(1) (1995), 135--157.

\bibitem{R}E. Roure Perdices,
Restricted weak type extrapolation of multi-variable operators and related topics,
Doctoral Dissertation, (2019).

\bibitem{Sawyer}E.T. Sawyer,
A characterization of a two-weight norm inequality for maximal operators,
Studia Math. 75(1) (1982), 1--11.

\bibitem{Sawyer1}E. Sawyer,
A weighted weak type inequality for the maximal function,
Proc. Amer. Math. Soc. 93(4) (1985), 610--614.

\bibitem{TZ}L. Tang and Q. Zhang,
Variation operators for semigroups and Riesz transforms acting on weighted $L^p$ and BMO spaces in the Schr\"{o}dinger setting,
Rev. Mat. Complut. 29(3) (2016), 559--621.

\bibitem{TV}S. Treil and A. Volberg,
Entropy conditions in two weight inequalities for singular integral operators,
Adv. Math. 301(5) (2016), 499--548.

\bibitem{WSS}Y. Wen, Q. Sun and J. Sun,
Weighted estimates for square functions associated with operators,
Math. Nachr. 296(8) (2023), 3725--3739.

\bibitem{WW}Y. Wen and H. Wu,
Quantitative weighted bounds for variation operators associated with heat semigroups in the Schr\"{o}dinger setting,
Rocky Mountain J. Math. 53(5) (2023), 1645--1656.

\bibitem{WWZ}Y. Wen, H. Wu and J. Zhang,
Weighted variation inequalities for singular integrals and commutators,
J. Math. Anal. Appl. 485(2) (2020), 16pp.

\bibitem{ZhYa}J. Zhang and D. Yang,
Quantitative boundedness of Littlewood-Paley functions on weighted Lebesgue spaces in the Schr\"{o}dinger setting,
J. Math. Anal. Appl. 484(2) (2020), 26pp.
\end{thebibliography}
\end{document}